\documentclass[10pt]{amsart}
\usepackage[margin= 1in]{geometry}
\usepackage{amsmath,amsthm,verbatim,amssymb,amsfonts,amscd,  graphicx}
\usepackage{xcolor,slashed}
\usepackage{soul}
\usepackage{graphics}

\usepackage{euscript, enumerate}
\usepackage{url,hyperref}

\newcommand{\e}{\epsilon}
\newcommand{\eps}{\varepsilon}

\newcommand{\be}{\begin{equation}}
\newcommand{\ba}{\begin{aligned}}
\newcommand{\bee}{\begin{equation*}}
\newcommand{\ee}{\end{equation}}
\newcommand{\ea}{\end{aligned}}
\newcommand{\eee}{\end{equation*}}
\newcommand{\bea}{\begin{equation} \begin{aligned} }
\newcommand{\eea}{\end{aligned}\end{equation} }

\theoremstyle{plain}
\newtheorem{theorem}{Theorem}[section]

\newtheorem{corollary}[theorem]{Corollary}

\newtheorem{lemma}[theorem]{Lemma}
\newtheorem{proposition}[theorem]{Proposition}

\theoremstyle{remark}

\newtheorem{remark}[theorem]{Remark}

\theoremstyle{definition}
\newtheorem{definition}[theorem]{Definition}

\numberwithin{equation}{section}

\begin{document}
\title{Asymptotics for slowly converging evolution equations}

\author{Beomjun Choi}
\address{Department of Mathematics, POSTECH, Pohang, 37673, Republic of Korea}
\email{bchoi@postech.ac.kr}

\author{Pei-Ken Hung}
\address{School of Mathematics, University of Minnesota, Minneapolis, MN 55455, USA}
\email{pkhung@umn.edu}

\begin{abstract} We investigate slowly converging solutions for non-linear evolution equations of elliptic or parabolic type. These equations arise from the study of isolated singularities in geometric variational problems. Slowly converging solutions have previously been constructed assuming the Adams-Simon positivity condition. In this study, we identify a necessary condition for slowly converging solutions to exist, which we refer to as the Adams-Simon non-negativity condition. Additionally, we characterize the rate and direction of convergence for these solutions. Our result partially confirms Thom's gradient conjecture in the context of infinite-dimensional problems.
%


  \end{abstract}
\maketitle

\section{Introduction}

Analyzing the behavior of solutions near singularities is essential to understand geometric equations such as minimal surfaces, harmonic maps and mean curvature flows. Many open problems reduce to questions on the singularity formation and its asymptotics.  In pioneering works of Leon Simon \cite{S0,S}, the idea of using {\L}ojasiewicz gradient inequality \cite{Loj} from real algebraic geometry was introduced for the first time, and the uniqueness of blow-ups was shown for a class of elliptic and parabolic equations. This uniqueness shows that the solution converges to the unique tangent cone or tangent flow as it approaches a singular point or infinity.

\bigskip 

A natural subsequent question is to investigate the rate of convergence and the next order asymptotics that describes the difference between the solution and the limit. This often serves as a crucial starting point for further analysis. For example, recent progresses on the classification of ancient solutions to geometric flows \cite{ADS}\cite{DuHaslhofer}\cite{CDDHS}\cite{ChoiMantoulidis} and complete non-compact solutions to minimal surface \cite{SS} are based on higher order asymptotics and its improvement.  For the singularity formation in parabolic problem, the higher order asymptotics at a singularity gives structural results on the singularity set in a neighborhood \cite{SunXue}\cite{Gang}. As the uniqueness of blow-ups implies the second differentiability of arrival time \cite{CM2015,CM2018}, the higher order asymptotics and the convergence rate have a strong relation to further regularity of arrival time \cite{MR2200259}\cite{MR2383930}.

\bigskip

The convergence rate and direction (i.e., the secant at the limit) are mostly understood when the solution converges at an exponential rate. This is because when the solution decays exponentially, the equation is well-approximated by its linearization. However, without the integrability of the limit, it is possible to have solutions that converge algebraically slowly. In \cite{AS}, Adams-Simon discovered a sufficient condition (later called the Adams-Simon positivity condition or simply $AS_p$ condition) on the limit so that they constructed a slowly converging solution to elliptic equations of the form \eqref{equ:main}. Carlotto-Chodosh-Rubinstein \cite{CCR15} found explicit examples of critical points of normalized Yamabe functional with the $AS_p$ condition and constructed normalized Yamabe flows converging slowly. As questioned in \cite[p.1533]{CCR15}, it is of great interest to understand the general behavior of slowly converging solutions. More precisely, one can ask whether such a solution must satisfy the Adams-Simon positivity condition and whether the higher order asymptotics follows the ansatz used in the construction of \cite{AS} and \cite{CCR15}.

\bigskip

We answer this question by showing that the Adams-Simon `non-negativity' condition is necessary for slowly converging solutions to exist. Moreover, when a positivity is satisfied, the convergence rate and the higher asymptotics agree with those of previously constructed examples. (See Theorem \ref{thm:general_e} and \ref{thm:general_p} for detailed statements.) We concern with the elliptic and parabolic equations of forms
\begin{equation}\label{equ:main}
 u''-mu'+\mathcal{M}_\Sigma u=N_1(u),
\end{equation} and
\begin{equation}\label{equ:parabolic}
u'-\mathcal{M}_\Sigma u=N_2(u).
\end{equation}
Here, $\mathcal{M}_\Sigma$ is the Euler-Lagrange operator of an analytic functional $\mathcal{F}_\Sigma$ as in \eqref{eq-F}-\eqref{eq-M}, $m$ is a nonzero constant, and we assume $N_1(u)$ and $N_2(u)$ satisfy the structure \eqref{equ:N}. The presence of nonzero constant $m$ accounts that the equation~\eqref{equ:main} appears in the study of the isolated singularity or the asymptotics near the infinity for geometric variational problems. In particular, minimal surfaces, harmonic maps \cite{S0,S,AS} and $G_2$ manifolds \cite{Chen} fall into this category. Equation~\eqref{equ:parabolic} models a geometric flow which converges to a stationary solution when $t$ goes to $+\infty$ or $-\infty.$ For instance, the (rescaled) mean curvature flow (see Appendix~\ref{sec:B}) and the harmonic map heat flow \cite{ChoiMantoulidis} can be described by \eqref{equ:parabolic}. Our results also apply to the normalized Yamabe flow. See Remark~\ref{rmk:cases} (3) for more details.

\bigskip

The main results of the paper, Theorems~\ref{thm:general_e}-\ref{thm:general_exponential_p}, can be placed in line as a generalization of Thom's gradient conjecture which we discuss in this paragraph. Let  $x=x(t)\in \mathbb{R}^n $ be a solution to the gradient flow $x'=-\nabla f(x)$ with an analytic potential $f:\mathbb{R}^n\to \mathbb{R}$ and suppose  $x$ converges to the origin as $t\to \infty$. In \cite{MR1067383}, Thom conjectured that the secant $ {x}/{|x|}$ should converge to a direction  $\theta_0 \in   \mathbb{S}^{n-1}$ as $t\to \infty$.  Kurdyka-Mostowski-Parusi\'nski \cite{KMP} settled this conjecture by showing a stronger result that the secant has a finite length in $\mathbb{S}^{n-1}$. However, the precise convergence rate is not yet completed revealed. If $\theta_0$ does not belong to $\ker \, (\nabla^2 f(0))$, one obtains that $\theta_0$ has to be an eigenvector of the hessian for a positive eigenvalue $\lambda$, and the solution decays with higher order asymptotics $|x(t)-ce^{-\lambda t} \theta_0| = o (e^{-\lambda t}) $ for some $c\in \mathbb{R}$. If $\theta_0$ belongs to the kernel of hessian, however, we merely know that the convergence takes place between algebraic rates $ t^{-\alpha_1} \lesssim  |x| \lesssim  t^{-\alpha_2 }$ for some $0<\alpha_2  \le \alpha_1\le 1 $. It is unknown if each solution has a specific convergence rate $|x|\sim t^{-\gamma  }$ some $\gamma >0$.

\bigskip

Let us briefly summarize the main motive and results of the paper. In the slow decaying regime, the equation becomes well-approximated by a gradient flow on the kernel of $\mathcal{L}_{\Sigma }$, the linearization of $\mathcal{M}_{\Sigma}$ at $u\equiv0$. The potential $f:  \ker \mathcal{L}_\Sigma \approx \mathbb{R}^J\to \mathbb{R}$ is given by the pull-back of $\mathcal{F}_{\Sigma}$ through Lyapunov-Schmidt reduction (see Proposition \ref{pro:AS}). Near the origin, the flow dynamics shall be determined by the first non-constant homogeneous polynomial appearing in the expansion of potential $f$. Let us denote this polynomial by $f_p$. We reserve $p\ge 3$ to denote the degree of this polynomial and we call it the order of integrability. When the gradient flow $x'=-\nabla f_p(x)$ is considered, one readily finds a radial solution converging to the origin whenever there exists a critical point of $\hat f_p:= f_p \vert_{\mathbb{S}^{J-1}}$ with positive critical value.  The slowly decaying solutions of \cite{AS} and \cite{CCR15} are obtained as perturbations of such radial ansatzes. 

\bigskip 

In Theorem \ref{thm:general_e} and \ref{thm:general_p}, we prove slowly decaying solution $u(\omega,t)$ has a dichotomy on its convergence rate that either $t^{\frac{1}{p-2}}\Vert u(\cdot,t)\Vert _{L^2}$ converges to a positive number or diverges to infinity. In the first case, we further show $u(\cdot,t)/\Vert u(\cdot,t)\Vert_{L^2}$ converges smoothly to a critical point of $\hat f_p$, say $v\in \ker \mathcal{L}_{\Sigma}$, with positive critical value. This shows the secant $u(\cdot,t)/\Vert u(\cdot,t)\Vert_{L^2}$ has a limit and confirms Thom's gradient conjecture for the concerned case. Note that one may interpret the theorem as the higher order asymptotics $u(t) = c  t^{-\frac{1}{p-2}} v + o(t^{-\frac{1}{p-2}})$, where $c$ is a constant of $p$ and $ f_p(v)$. Next, the second alternative shows a solution possibly decays even at a slower rate than $t^{-\frac{1}{p-2}}$. As seen in the classical gradient flow of potential $f=x_1^4+x_2^8$ on $\mathbb{R}^2$ (here, $p=4$ but a solution may decay at rate $t^{-\frac{1}{8-2}}$), the second alternative can actually take place. We show this can only occur if $\hat f_p$ admits critical point(s) of zero critical value and the secant $u(\cdot,t)/\Vert u(\cdot,t)\Vert_{L^2}$ accumulates on those critical point(s) as $t\to \infty$. It remains open whether the secant $u(\cdot,t)/\Vert u(\cdot,t)\Vert_{L^2}$ has a unique limit in the second case, and thereby, Thom's conjecture holds for all slowly decaying solutions. We would like to point out that Theorem \ref{thm:general_p} is also novel for finite-dimensional gradient flows. For a finite-dimensional gradient flow given by $x'=-\nabla f(x)$, the theorem states that there exists an integer $p\ge 3$ which depends only on $f$, such that any slowly converging solution satisfies either $t^{\frac{1}{p-2}}|x| = c (1+o(1))$ for some $c\in (0,\infty)$, or $t^{\frac{1}{p-2}}|x|\to \infty$.

\bigskip 

In Theorem \ref{thm:general_exponential_e} and \ref{thm:general_exponential_p}, we show the higher order asymptotic behavior of fast (exponentially) decaying solutions for elliptic and parabolic equations, respectively. This type of result has been expected among researchers and has been obtained for specific problems such as \cite{SS}\cite{ChoiMantoulidis}. Nevertheless, we could not find proper literature covering the general forms \eqref{equ:main} and \eqref{equ:parabolic}, so we provide a proof in Section~\ref{sec:exp}. The elliptic equation \eqref{equ:main} can be viewed as a perturbation of the second order ODE $u''-mu'+\mathcal{L}_\Sigma u=0$. Suppose $\mathcal{L}_\Sigma$ has an eigenvalue larger than $4^{-1}m^2$, a solution might oscillate while decaying exponentially. This type of solutions has been constructed for minimal graphs over Simons cones in $\mathbb{R}^4$ and $\mathbb{R}^6$. See \cite[Remark 1.21]{MR3838575} and \cite{BGG} for more details. For this reason, the original form of Thom's gradient conjecture is not true in the elliptic problem. Moreover, if $4^{-1}m^2$ is an eigenvalue of $\mathcal{L}_\Sigma$ and $m<0$, a resonance might occur and result in a solution that decays at a rate $t e^{2^{-1}m t}$. Note \cite{HS} considered this possibility for minimal graphs over stable (but not strictly stable) minimal cones. 
\bigskip

This paper is organized as follows. In Section \ref{sec:2}, we introduce the notation, condition, spectral property of linearized operator $\mathcal{L}_\Sigma$, and main theorems. In Section \ref{sec:1storderode}, we set up the elliptic problem \eqref{equ:main} as a first order ODE system on function spaces. In Section \ref{sec:exp}, Theorem \ref{thm:general_exponential_e} and Theorem \ref{thm:general_exponential_p} for exponentially decaying solutions to elliptic and parabolic equations are proved, respectively. In Section \ref{sec:elliptic}, we show in Proposition~\ref{prop-neutral-dynamics} that slowly converging solutions to \eqref{equ:main} are governed by a finite-dimensional gradient flow with a small perturbation. The parabolic analogue, Proposition~\ref{prop-neutral-dynamics-p1}, is proved in Section~\ref{sec:slowparabolic}. In Section~\ref{sec:gradient}, we complete the proofs for Theorems~\ref{thm:general_e} and \ref{thm:general_p} by analyzing a finite-dimensional gradient flow with a small perturbation, using a version of the {\L}ojasiewicz argument motivated by \cite{KMP}. Appendix~\ref{sec:A} contains auxiliary tools we need in the paper. In Appendix~\ref{sec:B}, we show that (rescaled) mean curvature flows can be written in the form \eqref{equ:parabolic}. 


\section{Preliminary}\label{sec:2}
Let us introduce our setting to study \eqref{equ:main} and \eqref{equ:parabolic} and state the main results. Let $\Sigma$ be a closed $n$-dimensional Riemannian manifold and $d\mu$ be a smooth volume form on $\Sigma$ which is mutually absolutely continuous with respect to the volume form induced by the metric. Let $\mathbf{V}\to \Sigma$ be a smooth vector bundle equipped with a smooth inner product $\left\langle\cdot,\cdot\right\rangle$. For $u\in\mathbf{V}$, we write $|u|^2:=\left\langle u,u \right\rangle$. For $\omega\in\Sigma$, we denote by $\mathbf{V}_\omega$ the fiber over $\omega$. Let $\slashed{\nabla}$ be a connection on $\mathbf{V}$ which is compatible with the inner product.

We denote by $L^2(\Sigma;\mathbf{V})$ the space of $L^2$-sections of $\mathbf{V}$ with respect to $d\mu$. Namely, a section $u$ belongs to $L^2(\Sigma;\mathbf{V})$ provided
\begin{align*}
\|u\|^2_{L^2}:=\int_\Sigma |u|^2\, d\mu <\infty.
\end{align*}
For $\ell \in\mathbb{N}_0$, we denote by $H^\ell(\Sigma;\mathbf{V})$ the collection of sections that satisfies
\begin{align*}
\|u\|^2_{H^\ell}:= \sum_{i=0}^\ell \int_\Sigma \left|\slashed{\nabla}^i u  \right|^2\, d\mu  <\infty.
\end{align*}
For $-\infty< a<b\leq \infty$, we define $Q_{a,b}:=\Sigma\times [a,b)$ and equip $Q_{a,b}$ with the product metric. We denote by $\widetilde{\mathbf{V}}$ the pull back bundle of $\mathbf{V}$ through the projection $  Q_{a,b}\to \Sigma$. For $u\in C^s(Q_{a,b};\widetilde{\mathbf{V}})$ and $t\in [a,b)$, $u(t)\in C^s(\Sigma;\mathbf{V})$ is defined through $u(t)(\omega):=u(\omega,t)$. Its $H^\ell$-norm is denoted by
\begin{equation*}
\|u\|_{H^\ell}(t):= \|u(t) \|_{H^\ell}.
\end{equation*} 
Also,
\begin{align}\label{def:Cs}
\|u\|_{C^s}(t):=\sup_{\omega\in\Sigma}\sup_{k+\ell\leq s}\left| \frac{\partial^k}{\partial t^k}\slashed{\nabla}^\ell u(\omega,t) \right|.
\end{align}
Here, note that the norms of time derivatives are included  in the definition of $\Vert \cdot \Vert _{C^s}  (t)$.
We often use $u'$ as an abbreviation of $\frac{\partial u }{\partial t}$ .

\vspace*{0.3cm}

Our assumption on $\mathcal{M}_\Sigma$ is almost identical to the one in \cite{S0}. The only difference is that the volume form, $d\mu$, is not necessarily the one induced from the Riemannian metric. Let $\mathcal{F}_\Sigma$ be a functional defined for $u\in C^{1}(\Sigma;\mathbf{V})$ by
\begin{align}\label{eq-F}
\mathcal{F}_\Sigma(u) :=\int_{\Sigma} F(\omega,u,\slashed{\nabla}u)\, d\mu.
\end{align}
Here the integrand $F$ satisfies the following:
\begin{enumerate}
\item $F=F(\omega,z,p)$ is a smooth function defined on an open set of $\mathbf{V}\times (T\Sigma\otimes \mathbf{V})$ that contains the zero section.
\item For each $\omega\in\Sigma$, $F(\omega,\cdot,\cdot)$ is analytic on $\mathbf{V}_\omega\times (T_\omega\Sigma\otimes \mathbf{V}_\omega)$.
\item $F$ satisfies the Legendre-Hadamard ellipticity condition
\begin{align*}
D^2_pF(\omega,0,0)[\eta\otimes\xi,\eta\otimes\xi]\geq c|\eta|^2|\xi|^2,
\end{align*}
for $c>0$ independent of $\omega\in\Sigma$, $\eta\in T_\omega\Sigma$ and $\xi\in \mathbf{V}_\omega$.
\end{enumerate}
$\mathcal{M}_{\Sigma}$ is defined to be the negative Euler-Lagrange operator of $\mathcal{F}_\Sigma$. Namely, for any $\zeta\in C^{\infty}(\Sigma;\mathbf{V})$,
\begin{align}\label{eq-M}
\int_\Sigma \left\langle \mathcal{M}_\Sigma( u),\zeta \right\rangle \, d\mu=-\frac{d}{ds}\mathcal{F}_\Sigma(u+s\zeta)|_{s=0}.
\end{align}
We further assume $0$ is a critical point of $\mathcal{F}_\Sigma$. Namely, we assume $\mathcal{M}_\Sigma(0)=0$. We denote by $\mathcal{L}_{\Sigma}$ the linearization of $\mathcal{M}_\Sigma$ at $0$. It is clear that the difference between $\mathcal{M}_{\Sigma} (u)$ and $\mathcal{L}_{\Sigma} u$ is a quadratic term of the form 
\begin{equation}\label{equ:quasilinear}
 \mathcal{M}_\Sigma u-\mathcal{L}_{\Sigma} u =\sum_{j=0}^2 c_j\cdot\slashed{\nabla}^j u,
\end{equation} 
where $c_j=c_j(\omega, u,\slashed{\nabla}u)$ are smooth with $c_j(\omega,0,0)=0$. The Legendre-Hadamard condition implies $\mathcal{L}_{\Sigma}$ is elliptic. In particular, there exists a constant $C>0$ such that for any $u\in C^2(\Sigma;\mathbf{V})$, 
\begin{align}\label{equ:H1}
\int_\Sigma -\left\langle \mathcal{L}_{\Sigma} u,u \right\rangle+C|u|^2 \, d\mu\approx \|u\|^2_{H^1}.
\end{align}
Furthermore, $\mathcal{L}_{\Sigma}$ is self-adjoint with respect to $d\mu$. This can be seen from
\begin{align*}
\int_{\Sigma} \left\langle \mathcal{L}_{\Sigma} u, v \right\rangle\, d\mu=-\frac{\partial^2}{\partial s_1\partial s_2}\mathcal{F}_\Sigma (s_1u+s_2v)\bigg|_{s_1=s_2=0}. 
\end{align*}

 We suppose $N_1(u)$ and $N_2(u)$ in \eqref{equ:main} and \eqref{equ:parabolic} are of the form
\begin{equation}\label{equ:N}
\begin{split}
N_1(u)=&a_1\cdot D u'+a_2\cdot u'+a_3\cdot\mathcal{M}_\Sigma (u),\\
 N_2(u)= & b_1\cdot\mathcal{M}_\Sigma (u). 
\end{split}
\end{equation}
Here $a_i=a_i(\omega, u,\slashed{\nabla}u,u')$ and $b_1=b_1(\omega,u,\slashed{\nabla}u)$ are smooth with $a_i(\omega,0,0,0)=0$ and $b_1(\omega,0,0)=0$; $D=\left\{ \frac{\partial}{\partial t},\slashed{\nabla} \right\}$. 

 Let $\lambda_1\geq \lambda_2\geq \dots$ be the eigenvalues of $\mathcal{L}_{\Sigma}$ and $\varphi_1, \varphi_2,\dots$ be the corresponding eigensections which form a complete orthonormal basis of $L^2(\Sigma;\mathbf{V})$. We separate $\mathbb{N}$ into four parts according to the eigenvalues.
\begin{equation}\label{def:I123}
\begin{split}
I_1:= &\{i\in\mathbb{N}\, :\, \lambda_i>4^{-1}m^2\},\ 
I_2:=  \{i\in\mathbb{N}\, :\, \lambda_i=4^{-1}m^2\},\\ 
I_3:= & \{i\in\mathbb{N}\, :\, \lambda_i=0\},\
I_4:= \mathbb{N}\setminus (I_1\cup I_2 \cup I_3).
\end{split}
\end{equation}
Note that $I_1$, $I_2$ and $I_3$ are (possibly empty) finite sets and $\ker \mathcal{L}_{\Sigma}$ is spanned by $\{\varphi_i\}_{i\in I_3}$. Let $J$ be the cardinality of $I_3$, the dimension of  $\ker \mathcal{L}_{\Sigma}$. This implies $I_3=\{\iota+ 1,\iota+2,\dots, \iota+J\}$ for some $\iota\in\mathbb{N}_0$. 

We denote by $\Pi^T$ and $\Pi^\perp$ the orthogonal projection of $L^2(\Sigma;\mathbf{V})$ to $\ker \mathcal{L}_{\Sigma}$ and $\left(\ker \mathcal{L}_{\Sigma}\right)^\perp$ respectively.  The following is a version of the implicit function theorem. See \cite[\S 3]{S1}.
\begin{proposition} [Lyapunov--Schmidt reduction]\label{pro:AS} 
Let $B_\rho$ be the open ball of radius $\rho$ in $L^2(\Sigma;\mathbf{V})$. There exist $\rho>0$ and a map \[H: \ker \mathcal{L}_{\Sigma} \cap B_\rho \rightarrow C^\infty (\Sigma;\mathbf{V} )\cap (\ker \mathcal{L}_{\Sigma}) ^\perp \] such that the following statements hold. First, for any $k\geq 2$ and $\alpha\in (0,1)$, $H$ is an analytic map from $\ker \mathcal{L}_{\Sigma} \cap B_\rho$ to $C^{k,\alpha}(\Sigma;\mathbf{V})$. Second,  
\[H(0)=0,\, DH(0)=0. \]
Lastly,
\[\begin{cases} \begin{aligned} \Pi ^\perp   \mathcal{M}_{\Sigma} (v+H(v) ) &= 0,\\ 
\Pi ^T \mathcal{M}_{\Sigma} (v+H(v) )& = -\nabla f (v).
\end{aligned} \end{cases}\]
Here $f:\ker \mathcal{L}_{\Sigma}  \cap B_\rho \to \mathbb{R}$ is given by $f(v):= \mathcal{F}_\Sigma (v+H(v)).$ 
\end{proposition}
The function $f$ plays a crucial role in this paper and we will call it \textbf{reduced functional}. The reduced functional $f$ is real analytic with $\nabla f(0)=0$ and $\nabla^2 f(0)=0$. We may view $f$ as an analytic function defined on an open ball in $\mathbb{R}^J$ through the identification
\begin{align}\label{equ:identification}
(x_1,x_2,\dots, x^J)\mapsto \sum_{j=1}^J x_j\varphi_{\iota+j}.
\end{align} 

Let us review the integrable condition. The kernel $\ker \mathcal{L}_{\Sigma}$ is called integrable if for any $v\in\ker \mathcal{L}_{\Sigma}$, there exists a family $\{v_s\}_{s\in (0,1)}\subset C^2(\Sigma;\mathbf{V})$ such that $v_s\to 0$ in $C^2(\Sigma;\mathbf{V})$, $\mathcal{M}_\Sigma (v_s)\equiv 0$ and $\lim_{s\to 0} v_s/s=v$ in $L^2(\Sigma;\mathbf{V})$. It is well-known \cite[Lemma 1]{AS} that $\ker \mathcal{L}_{\Sigma}$ is integrable if and only if the reduced functional $f$ is a constant. Moreover, if integrable condition is satisfied, any decaying solution to \eqref{equ:main} or \eqref{equ:parabolic} decays exponentially \cite{AS,CCR15}. For this reason, whenever non-exponentially decaying solution is considered, the integrable condition should necessarily \textbf{fails}. Namely, the reduced functional $f$ is not a constant function. In particular, there exists an integer $p\geq 3$ such that
\begin{equation}\label{def:fp}
f=f(0)+\sum_{j\geq p} f_j,
\end{equation}
where $f_j$ are homogeneous polynomials with degree $j$ and $f_p\not\equiv 0$. This integer $p$ is called the order of integrability \cite{CCR15}.

As explained in Introduction, the gradient flow of $f_p$ has a dominant role in the asymptotic behavior. Let $\hat{f}_p$ be the restriction of $f_p$ on $\{ w\in \ker \mathcal{L}_{\Sigma}  \, :\, \Vert w\Vert _{L^2}=1\} \approx \mathbb{S}^{J-1}$. Consider the critical points of $\hat{f}_p$:
\begin{equation}\label{def:C}
\mathbf{C}:=\left\{ w\in \ker \mathcal{L}_{\Sigma}\, :\, \Vert w\Vert_{L^2}=1\ \textup{and}\ w\ \textup{is a critical point of}\ \hat{f}_p \right\}.
\end{equation}
If $w\in\mathbf{C}$ satisfies $ f_p(w)>0$, then one checks that 
\[x(t)=   [ p(p-2)f_p(w) (t+c) ] ^{-\frac{1}{p-2}} w\]
 becomes a radial solution to the flow $x'= -\nabla f_p(x)$. The higher order asymptotics in Theorem \ref{thm:general_e} and \ref{thm:general_p} will be modeled on such solutions.

\begin{definition}[Adams-Simon conditions. c.f. (4.1) in \cite{AS}] We say $\Sigma$ satisfies the Adams-Simon non-negativity condition for  \eqref{equ:parabolic} if there exists $w\in\mathbf{C}$ such that  ${f}_p(w)\geq 0$. 
We say $\Sigma$ satisfies the Adams-Simon non-negativity condition for \eqref{equ:main} if there exists $w\in\mathbf{C}$ such that  $m^{-1}{f}_p(w)\geq 0$. 

In both equations, the Adams-Simons positivity conditions are defined similarly by requiring that the critical values are positive. 
\end{definition}

Let us state main results concerning the asymptotic behavior and the convergence rate of decaying solutions. We begin with the elliptic equation~\eqref{equ:main}.  

\begin{theorem} [slow decay in elliptic equation]  \label{thm:general_e}
Let $u\in C^\infty(Q_{0,\infty},\widetilde{ \mathbf{V}})$ be a  solution to \eqref{equ:main} with $\|u\|_{C^1}(t)=o(1)$ that does not decay exponentially as $t\to \infty$. Then $\Sigma$ satisfies the Adams-Simon non-negativity condition for \eqref{equ:main}. Moreover, one of the following alternatives holds:   
\begin{enumerate}
\item  We have
\begin{align*}
\lim_{t\to\infty} t^{1/(p-2)}\Vert u(t)\Vert_{L^2}=\beta \in (0,\infty).  
\end{align*}
Moreover,
\begin{align*}
\lim_{t\to \infty  }{u(t)}/{\Vert u(t)\Vert _{L^2}}=w\ \textup{in}\ C^\infty(\Sigma;\mathbf{V}),
\end{align*}
where $w \in \mathbf{C}$ with $m^{-1}\hat f_p (w)=  1/ ({p (p-2) \beta ^{p-2}})>0	$ .
\item We have
\begin{align*}
\lim_{t\to\infty} t^{1/(p-2)}\Vert u(t)\Vert_{L^2}=\infty.
\end{align*}
Moreover,
$$\lim_{t\to\infty} \textup{dist}_{C^k}\bigg   ( u(t)/\Vert u(t) \Vert_{L^2}\  , \ \mathbf{C}\cap \{w : \hat f_p(w)= 0\}\bigg )=0\ \textup{for all}\  k \in \mathbb{N}.$$
\end{enumerate}
\end{theorem}
\begin{theorem}[fast decay in elliptic equation]\label{thm:general_exponential_e}
Let $u\in C^\infty(Q_{0,\infty},\widetilde{ \mathbf{V}})$ be a  solution to \eqref{equ:main} with $\|u\|_{C^1}(t)=O(e^{-\varepsilon t})$ for some $\varepsilon>0$ as $t\to \infty$. If $u$ is not the zero section, then one of the following alternatives holds: 

\begin{enumerate}
\item  There exists an eigenvalue $\lambda<4^{-1}m^2$ of $\mathcal{L}_{\Sigma}$ such that 
\begin{align*}
\gamma^+<0\ \textup{and}\  \lim_{t\to \infty} e^{-\gamma^+t} \Vert u(t)\Vert _{L^2}\in (0,\infty), 
\end{align*}
or
\begin{align*}
\gamma^-<0\ \textup{and}\ \lim_{t\to \infty} e^{-\gamma^-t} \Vert u(t)\Vert _{L^2}\in (0,\infty) .
\end{align*}
Here $\gamma^\pm=2^{-1}m\pm\sqrt{4^{-1}m^2-\lambda}. $ 
Moreover, for some eigensection $w$ with $\mathcal{L}_{\Sigma} w =\lambda w$,  
\begin{align*}
\lim_{t\to \infty  }{u(t)}/{\Vert u(t)\Vert _{L^2}}=w\ \textup{in}\ C^\infty(\Sigma;\mathbf{V}).
\end{align*} 

\item We have
\begin{align*}
 m<0\ \textup{and}\    \lim_{t\to \infty}  t^{-1}e^{-2^{-1}m t} \Vert u(t)\Vert _{L^2}\in (0,\infty).
\end{align*}
Moreover, for some eigensection $w$ with $\mathcal{L}_{\Sigma} w =4^{-1}m^2 w$,  
\begin{align*}
\lim_{t\to \infty  }{u(t)}/{\Vert u(t)\Vert _{L^2}}=w\ \textup{in}\ C^\infty(\Sigma;\mathbf{V}).
\end{align*}
\item We have
\begin{align*}
 m<0\ \textup{and}\   \limsup_{t\to \infty} e^{-2^{-1}m t} \Vert u(t)\Vert _{L^2} \in (0,\infty).
\end{align*}
Moreover, there exist $w_i\in\mathbb{C}$ for $i\in I_1$ and $c_{i}\in\mathbb{R}$ for $i\in  I_2$ such that
\begin{align*}
\lim_{t\to \infty  }\left( e^{-2^{-1}mt}u(t) -\sum_{i\in I_1}\textup{Re} \left(w_ie^{\mathbf{i}\beta_i t}  \right)\varphi_i-\sum_{i\in I_2}c_i \varphi_i\right) =0\   \textup{in}\ C^\infty(\Sigma;\mathbf{V}).
\end{align*}
Here $\beta_i=\sqrt{\lambda_i-4^{-1}m^2}$ and $\mathbf{i}=\sqrt{-1}$. 

\end{enumerate}

\end{theorem}

Next, we consider decaying solutions to the parabolic equation~\eqref{equ:parabolic}. 

\begin{theorem}[slow decay in parabolic equation]\label{thm:general_p}
Let $u\in C^\infty(Q_{0,\infty},\widetilde{ \mathbf{V}})$ be a solution to \eqref{equ:parabolic} with $\|u\|_{H^{n+4}}(t)=o(1)$ that does not decay exponentially as $t\to \infty$. Then $\Sigma$ satisfies the Adams-Simon non-negativity condition for \eqref{equ:parabolic}. Moreover, one of the following alternatives holds: 
\begin{enumerate}
\item We have
\begin{align*}
\lim_{t\to\infty} t^{1/(p-2)}\Vert u(t)\Vert_{L^2}=\beta \in (0,\infty).  
\end{align*}
Moreover,
\begin{align*}
\lim_{t\to \infty  }{u(t)}/{\Vert u(t)\Vert _{L^2}}=w\ \textup{in}\ C^\infty(\Sigma;\mathbf{V}),
\end{align*}
where $w \in \ker \mathcal{L}_{\Sigma} $ and $w$ is a critical point of $\hat f_p $ with $\hat f_p (w)=  1/ ({p (p-2) \beta ^{p-2}})>0	$ .
\item We have
\begin{align*}
\lim_{t\to\infty} t^{1/(p-2)}\Vert u(t)\Vert_{L^2}=\infty.
\end{align*}
Moreover,
$$\lim_{t\to\infty} \textup{dist}_{C^k}\bigg   ( u(t)/\Vert u(t) \Vert_{L^2}\  , \ \mathbf{C}\cap \{w : \hat f_p(w)= 0\}\bigg )=0\ \textup{for all}\  k \in \mathbb{N}.$$
\end{enumerate}  
\end{theorem}
\begin{theorem}[fast decay in parabolic equation]\label{thm:general_exponential_p}
Let $u\in C^\infty(Q_{0,\infty},\widetilde{ \mathbf{V}})$ be a solution to \eqref{equ:parabolic} with $\|u\|_{H^{n+4}}(t)=O(e^{-\varepsilon t})$ for some $\varepsilon>0$ as $t\to \infty$. Then there exists a negative eigenvalue $\lambda$ of $\mathcal{L}_{\Sigma}$ such that 
\begin{align*}
\lim_{t\to \infty} e^{-\lambda t} \Vert u(t)\Vert _{L^2}\in (0,\infty).
\end{align*}
Moreover, for some eigensection $w$ with $\mathcal{L}_{\Sigma} w =\lambda w$, 
\begin{align*}
\lim_{t\to \infty  }{u(t)}/{\Vert u(t)\Vert _{L^2}}=w\ \textup{in}\ C^\infty(\Sigma;\mathbf{V}).
\end{align*} 
\end{theorem}

{In fact, it is possible that no slowly converging solution exists even in the presence of non-integrable kernel. e.g., we refer \cite{choi2021translating}\cite{choisunclassification}. Theorem \ref{thm:general_e} and \ref{thm:general_p} provide a criterion for this non-existence result.}

\begin{corollary}
Let $u\in C^\infty(Q_{0,\infty},\widetilde{ \mathbf{V}})$ be a solution to \eqref{equ:main} with $\|u\|_{C^1}(t)=o(1)$ as $t\to \infty$. Suppose the Adams-Simon non-negative condition for \eqref{equ:main} fails. Then $u$ decays exponentially. 
\end{corollary}
\begin{corollary}\label{cor:pp}
Let $u\in C^\infty(Q_{0,\infty},\widetilde{ \mathbf{V}})$ be a solution to \eqref{equ:parabolic} with $\|u\|_{H^{n+4}}(t)=o(1)$ as $t\to \infty$. Suppose the Adams-Simon non-negative condition for \eqref{equ:parabolic} fails. Then $u$ decays exponentially.
\end{corollary}


\begin{remark}\label{rmk:cases}
We note that the main results, Theorems~\ref{thm:general_e}-\ref{thm:general_exponential_p}, can be generalized to cover other cases.
\begin{enumerate}
\item We may consider ancient solutions defined on $t\in (-\infty,0]$ which decay to zero when $t$ approaches minus infinity. For ancient solutions to the elliptic equation \eqref{equ:main}, we can simply perform a change of variable $t\mapsto -t$. For ancient solutions to the parabolic equation \eqref{equ:parabolic}, simple changes in the proofs yield Theorems~\ref{thm:general_p} and \ref{thm:general_exponential_p}.
\item We may consider the hyperbolic equation of the form
\begin{align*}
 -u''-mu'+\mathcal{M}_\Sigma u=N_1(u).
\end{align*}
The only difference from the elliptic equation \eqref{equ:main} is that, due to the lack of elliptic regularity, we need to assume $\Vert u \Vert_{C^s}(t)$ decays to zero for all $s\in\mathbb{N}$.
\item Due to volume normalization, the normalized Yamabe flow (NYF) involves a non-local term,  the total scalar curvature, in its speed. Nevertheless, our proof for Theorem~\ref{thm:general_p} can be modified to cover the NYF. Specifically, all parts of the proof, except for Lemma~\ref{lem:error_p}, remain the same. Lemma~\ref{lem:error_p} can be easily established with the explicit forms of the NYF.
\item The main results apply for the gradient flow on $\mathbb{R}^n$ by choosing $\Sigma$ a point, $d\mu$ a Dirac mass, and $V$ a trivial $\mathbb{R}^n$ bundle. Indeed, \eqref{equ:main} and \eqref{equ:parabolic} cover wider class of ODE systems which evolve under analytic potentials.
\end{enumerate}

\end{remark}

\bigskip

We finish this section with two lemmas related to the Lyapunov-Schmidt reduction map $H$ given in Proposition~\ref{pro:AS}. We decompose $u$ as follows: let $u^T=\Pi^T u$ and define $\tilde{u}^\perp$ by
\begin{align}\label{def:ut}
u=u^T+H(u^T)+\tilde{u}^\perp.
\end{align}
Since $u^T\in\ker \mathcal{L}_{\Sigma}$, $u^T$ can be written as a linear combination of $\{\varphi_{\iota+j}\}_{1\leq j\leq  J}$ as
\begin{equation}\label{def:xj}
u^T= \sum_{j=1}^J x_j\varphi_{\iota+j}.
\end{equation}

\begin{lemma}\label{lem:MSigma}
Let $\rho>0$ be the constant given in Proposition~\ref{pro:AS} and $u\in C^{2,\alpha}(\Sigma;\mathbf{V})$ be a section with $\| u^T \|_{L^2}\leq 2^{-1}\rho$. Let $x=(x_1,x_2,\dots,x_J)$ be given by \eqref{def:xj}. Then for any $\alpha\in (0,1)$, there exists a positive constant $C=C(\mathcal{M}_\Sigma,\alpha)$ such that
\begin{equation} 
\left| \mathcal{M}_\Sigma (u)+\nabla f(x)-\mathcal{L}_{\Sigma} \tilde{u}^\perp \right|\leq C \Vert u \Vert_{C^{2,\alpha}(\Sigma)}\Vert \tilde{u}^\perp  \Vert_{C^{2,\alpha}(\Sigma)} .
\end{equation} 
\begin{proof}
In the proof we use $C$ to represent a positive constant that depends on $\mathcal{M}_\Sigma$, $\alpha$, and its value may vary from one line to another.
Let $\bar{\mathcal{L}}_\Sigma$ be the linearization of $\mathcal{M}_\Sigma$ at $u^T+H(u^T)$. From Proposition~\ref{pro:AS} (Lyapunov-Schmidt reduction) and \eqref{def:ut},
\begin{align*}
&\left| \mathcal{M}_\Sigma (u)+\nabla f(x)-\mathcal{L}_{\Sigma} \tilde{u}^\perp \right|\\
 =&\left|    \mathcal{M}_{\Sigma}(u^T +H(u^T)+ \tilde u^\perp)-\mathcal{M}_{\Sigma} (u^T +H(u^T)) -\mathcal{L}_{\Sigma} \tilde{u}^\perp\right|\\
 \leq &|\bar{\mathcal{L}}_\Sigma  \tilde{u}^\perp -\mathcal{L}_{\Sigma} \tilde{u}^\perp  |+ |\mathcal{M}_{\Sigma}(u^T +H(u^T)+ \tilde u^\perp)-\mathcal{M}_{\Sigma} (u^T +H(u^T))-\bar{\mathcal{L}}_\Sigma  \tilde{u}^\perp   |. 
\end{align*}
Because $\mathcal{M}_\Sigma$ is an analytic map from $C^{2,\alpha}(\Sigma;\mathbf{V})$ to $C^{\alpha}(\Sigma;\mathbf{V})$, the above is bounded by  
\begin{align*}
C  \Vert u^T+H(u^T)\Vert_{C^{2,\alpha}(\Sigma)}\Vert \tilde u^\perp \Vert_{C^{2,\alpha}(\Sigma)}+C\Vert \tilde u^\perp \Vert^2_{C^{2,\alpha}(\Sigma)}
\end{align*}
From \eqref{def:xj} and the analyticity of $H$, there holds  $\Vert u^T+H(u^T)\Vert_{C^{2,\alpha}(\Sigma)}\leq C |x|\leq C \Vert u \Vert_{C^{2,\alpha}(\Sigma)}.$ 	Together with \eqref{def:ut}, we have $\Vert \tilde{u}^\perp \Vert_{C^{2,\alpha}(\Sigma)} \leq C  \Vert u \Vert_{C^{2,\alpha}(\Sigma)}.$ Hence the assertion holds.

\end{proof}
\end{lemma}
 
\begin{lemma}[boundedness of decomposition]\label{lem:utup} 
Let $\rho>0$ be the constant given in Proposition~\ref{pro:AS} and $u\in C^\infty (Q_{a,b};\widetilde{\mathbf{V}})$ be a smooth section with $\|u\|_{C^s}(t) \le M <\infty $. Then when $\|u^T\|_{L^2}(t)\leq 2^{-1}\rho$, 
\begin{align*}
\|u^T\|_{C^s}(t) +\|H(u^T)\|_{C^s}(t) + \|\tilde{u}^\perp\|_{C^s}(t) \le C  \|u\|_{C^s}(t),
\end{align*}for some constant $C=C(\mathcal{M}_\Sigma, M,s)$.
\begin{proof}
In the proof we use $C$ to represent a positive constant that depends on $\mathcal{M}_\Sigma$, $M$, $s$, and its value may vary from one line to another. Let $x(t)=(x_1(t),\dots, x_J(t))$ be the coefficients given by \eqref{def:xj} for $u^T(t)$. Then
\begin{align*}
\|u^T\|_{C^s}(t)\leq C \sum_{k=0}^s \left| \frac{d^k}{dt^k} x(t) \right|\leq C   \|u\|_{C^s}(t).
\end{align*}
Through \eqref{equ:identification}, we may view $H$ as a map from an open ball in $\mathbb{R}^J$ to $C^\infty(\Sigma;\mathbf{V})$. We abuse the notation and write $H(x(t))$ for $H(u^T)$. For $1\leq k\leq s$, 
\begin{align*}
\frac{\partial^k}{\partial t^k} H(x(t))=\sum_{i=1}^k\sum_{\substack{k_1+\dots+k_i=k\\k_j\geq 1}}  D^iH(x(t)) \left[  \frac{d^{k_1} x (t)}{dt^{k_1} },  \dots  ,\frac{d^{k_i} x (t)}{dt^{k_i} } \right].
\end{align*}

Let $\mathcal{B}_{i,\ell,\alpha}= \mathcal{L}(\left(\mathbb{R}^J\right)^{\otimes i}, C^{\ell,\alpha}(\Sigma;\mathbf{V}))$ be the Banach space of bounded linear maps from $\left(\mathbb{R}^J\right)^{\otimes i}$ to $ C^{\ell,\alpha}(\Sigma;\mathbf{V})$ equipped with the operator norm. From Proposition~\ref{pro:AS}, $D^iH$ is an analytic map from $B_\rho(0)\subset \mathbb{R}^J$ to $\mathcal{B}_{i,\ell,\alpha}$. In particular, the operator norm of $D^iH$ is bounded in $B_{2^{-1}\rho}(0)$. Therefore, provided $|x(t)|\leq 2^{-1}\rho$,
\begin{align*}
\left| \slashed{\nabla}^\ell\left[ D^iH(x(t))\left [ \frac{d^{k_1} x (t)}{dt^{k_1} },  \dots ,  \frac{d^{k_i} x (t)}{dt^{k_i} } \right ]\right]\right|\leq C \left|\frac{d^{k_1} x (t)}{dt^{k_1} } \right|  \dots  \left| \frac{d^{k_i} x (t)}{dt^{k_i} } 	\right| 
 \leq C   \|u\|_{C^s}(t).
\end{align*} 
We used the assumption $\|u\|_{C^s}(t)\le M$ in the second inequality. This ensures $ \|H(u^T)\|_{C^s}(t) \leq C  \|u\|_{C^s}(t).$ Then because of \eqref{def:ut}, $ \|\tilde{u}^\perp\|_{C^s}(t) \leq C \|u\|_{C^s}(t)$ holds.  
\end{proof}
\end{lemma}

\section{First order ODE system} \label{sec:1storderode}
In this section, we transform \eqref{equ:main} into a first order ODE system.  Let $u\in C^2(Q_{0,\infty},\widetilde{\mathbf{V}})$ be a solution to \eqref{equ:main}. By setting ${E}_1 (u)=N_1(u)- \mathcal{M}_\Sigma u+\mathcal{L}_{\Sigma} u$, we may rewirte \eqref{equ:main} as
\bea\label{equ:minimalODE}  u''- m  u'+   \mathcal{L}_{\Sigma} u = E_1(u).\eea
From \eqref{equ:quasilinear} and \eqref{equ:N}, the error term $E_1(u)$ has the structure
\begin{equation}\label{equ:Estr}
 {E}_1(u)=a_1\cdot Du'(t)+a_2\cdot u'(t)+\sum_{j=0}^2 a_{4,j}\cdot \slashed{\nabla}^j u,
\end{equation} 
where $a_{4,j}=a_{4,j}(\omega, u,\slashed{\nabla}u,u')$ are smooth with $a_{4,j}(\omega,0, 0,0)=0$. We aim to vectorize \eqref{equ:minimalODE} and view it as a first order ODE.  

Let us begin to set up some notion.
\begin{definition}\label{def:mathcal}
Let $\mathbf{L}$ be an operator from $H^2(\Sigma;\mathbf{V})\times H^1(\Sigma;\mathbf{V})$ to $H^1(\Sigma;\mathbf{V})\times L^2(\Sigma;\mathbf{V})$ given by
\begin{align*}
\mathbf{L}(v,w) := \left( 2^{-1}m v+
w , -\mathcal{L}_{\Sigma}v+4^{-1}m^2 v+2^{-1}m w
\right).
\end{align*}
\end{definition}
\begin{definition}\label{def:qE}
For a section $u\in C^2(Q_{0,\infty},\widetilde{\mathbf{V}})$, define
\begin{align*}
 q(u):= (u, u' - 2^{-1}mu),\ \mathcal{E}(u) := (0,E_1(u)). 
\end{align*} 
If there is no confusion, we often omit the argument and write $q$, $\mathcal{E}$ to denote $q(u)$ and $\mathcal{E}(u)$, respectively. Moreover,  $q(t)$ and $\mathcal{E}(t)$ denote the restriction of $q(u)$ and $\mathcal{E}(u)$ on $\{t\}\times\Sigma$, respectively.
\end{definition}

The equation \eqref{equ:minimalODE} can be rewritten as
\begin{align}\label{eq-vectoru}
 q'=\mathbf{L}q+\mathcal{E} .
\end{align}
We now identify the eigenvalues and eigensections of $\mathbf{L}$. For $i\in\mathbb{N}$, let 
$$\gamma^\pm _i :=  2^{-1}m  \pm \sqrt {4^{-1}m^2   -\lambda_i }.$$
Recall that $\mathbb{N}$ is divided into $\cup_{i=1}^4 I_i $ in \eqref{def:I123}. The eigenvalues $\gamma^\pm_i$ are real and different from $2^{-1}m$ if and only if $i\in I_3\cup I_4$. For $i\in I_3\cup I_4$, let 
\begin{equation}\label{def:psi+-}
\psi_i ^\pm :=   \bigg ( \frac{m}{m-2\gamma^\pm_i }\varphi_i ,  -\frac{m}{2}  \varphi_i\bigg ).
\end{equation}
From $\mathcal{L}_{\Sigma}\varphi=\lambda_i\varphi_i$, one can check that 
\begin{equation}\label{equ:L1}
 \mathbf{L}\psi^{\pm}_i=\gamma^\pm_i\psi^\pm_i.
\end{equation}
For $i\in I_1$, $\gamma^{\pm}_i$ are not real with the imaginary part $\beta_i= \sqrt{  \lambda_i-4^{-1}m^2 }>0$. Set
\begin{align*}
\psi_{i,1}:= (0, 2^{-1/2}m\varphi_i),\ \psi_{i,2}:= (2^{-1/2}m\beta_i^{-1}\varphi_i ,0).
\end{align*}   
Then \begin{equation}\label{equ:L2}
 \mathbf{L}\psi_{i,1}=2^{-1}m \psi_{i,1}+\beta_i\psi_{i,2},\ \mathbf{L}\psi_{i,2}=2^{-1}m \psi_{i,2}-\beta_i\psi_{i,1}. 
\end{equation}
For $i\in I_2$, let
\begin{align*}
\psi_{i,3}:=(0,2^{-1/2}m\varphi_i),\ \psi_{i,4}:=(2^{-1/2}m\varphi_i,0).
\end{align*}
Then 
\begin{equation}\label{equ:L3}
 \mathbf{L}\psi_{i,3}=2^{-1}m \psi_{i,3}+\psi_{i,4},\ \mathbf{L}\psi_{i,4}=2^{-1}m \psi_{i,4}. 
\end{equation}

Next, we introduce a bilinear form $G$. Let $G:( H^1(\Sigma;\mathbf{V})\times H^0(\Sigma;\mathbf{V})  )^2\to\mathbb{R}$ be defined by
\begin{align*}
G((v_1,w_1);(v_2,w_2)):= &2m^{-2}  \int_{\Sigma}\big( -\left\langle \mathcal{L}_{\Sigma} v_1-4^{-1}m^2v_1 ,v_2 \right\rangle+\left\langle w_1 ,w_2 \right\rangle \big)\, d\mu\\
 &+ 2m^{-2} \sum_{i\in I_1}  2  \beta^2_i \int_\Sigma \left\langle v_1, \varphi_i \right\rangle d\mu \int_\Sigma \left\langle v_2, \varphi_i \right\rangle d\mu  +2m^{-2}\sum_{i\in I_2}\int_\Sigma \left\langle v_1, \varphi_i \right\rangle d\mu \int_\Sigma \left\langle v_2, \varphi_i \right\rangle d\mu . 
\end{align*}
We denote $\|(v,w)\|^2_G:= G((v,w);(v,w))$. The positive-definiteness of $G$ will be justified in Lemma~\ref{lem:G} below. We write $\mathbf{L}^\dagger$ for the adjoint operator of $\mathbf{L}$ with respect to $G$. Namely,
\begin{align*}
G(\mathbf{L}(v_1,w_1);(v_2,w_2))=G((v_1,w_1);\mathbf{L}^\dagger(v_2,w_2)).
\end{align*} 
We define the collection of vectors 
\begin{equation}\label{def:B}
 \mathcal{B}:= \{\psi_{i,1}, \psi_{i,2}\}_{i\in I_1}\cup\{\psi_{i,3}, \psi_{i,4}\}_{i\in I_2} \cup\{\psi^+_i,\psi^-_i\}_{i\in I_3\cup I_4}. 
\end{equation}
\begin{lemma}[spectral decomposition of  $\mathbf{L}$]\label{lem:G}
\hfill
\begin{enumerate}
\item  The bilinear form $G$ is equivalent to the standard inner product. Namely,
\begin{equation}\label{equ:Gequiv}
  \|(v,w)\|_G \approx \Vert (v,w)\Vert_{H^1\times H^0}. 
\end{equation}
\item The collection $\mathcal{B}$ forms a complete $G$-orthonormal basis.
\item For $i\in I_3\cup I_4$, 
\begin{align}\label{equ:Ldagger1}
\mathbf{L}^\dagger \psi^\pm_i=\gamma^\pm_i\psi^\pm_i.
\end{align}
For $i\in I_1$, 
\begin{align}\label{equ:Ldagger2}
\mathbf{L}^\dagger \psi_{i,1}=2^{-1}m\psi_{i,1}-\beta_i\psi_{i,2} ,\ \mathbf{L}^\dagger \psi_{i,2}=2^{-1}m\psi_{i,2}+\beta_i\psi_{i,1}.
\end{align}
For $i\in I_2$, 
\begin{align}\label{equ:Ldagger3}
\mathbf{L}^\dagger \psi_{i,3}=2^{-1}m\psi_{i,3},\ \mathbf{L}^\dagger \psi_{i,4}=2^{-1}m\psi_{i,4}+\psi_{i,3}.
\end{align}
\end{enumerate}
\end{lemma}
\begin{proof}
We start to prove \eqref{equ:Gequiv}. By expressing $v$ as $v=\sum_{i=1}^\infty a_i \varphi_i$, \begin{align*}
G((v,0);(v,0))=2m^{-2}\left(\sum_{i\notin I_2 } \left| 4^{-1}m-\lambda_i \right| a_i^2+\sum_{i\in I_2 } a_i^2 \right) \geq c \sum_{i=1}^\infty a_i^2=c\|v\|^2_{L^2}, 
\end{align*} 
for some $c=c(\mathcal{M}_\Sigma,m)>0 $. In view of \eqref{equ:H1}, this implies \eqref{equ:Gequiv}. It is straightforward to check that vectors in $\mathcal{B}$ are $G$-orthonormal. Suppose $(v,w)\in H^1(\Sigma;\mathbf{V})\times H^0(\Sigma;\mathbf{V}) $ is $G$-orthogonal to every vector in $\mathcal{B}$. Note that for all $j\in \mathbb{N}$, $(0,\varphi_j)$ and $(\varphi_j,0)$ lie in the linear span of $\mathcal{B}$. This implies that 
\begin{align*}
0=2^{-1}m^{2} G((v,w);(0,\varphi_j))= \int_{\Sigma}\left\langle w,\varphi_j \right\rangle \, d\mu,
\end{align*}
and that
\begin{align*}
0=2^{-1}m^{2} G((v,w);(\varphi_j,0 ))= C_j  \int_{\Sigma}\left\langle v,\varphi_j \right\rangle \, d\mu.
\end{align*}
Here $C_j=1$ for $j\in I_2$ and $C_j=\left| 4^{-1}m-\lambda_j \right|$ for $j\notin I_2$. Therefore, $v=w=0$ and $\mathcal{B}$ is complete. Lastly, \eqref{equ:Ldagger1}-\eqref{equ:Ldagger3} follow from \eqref{equ:L1}-\eqref{equ:L3} and the $G$-orthogonality of $\mathcal{B}$.
\end{proof}

The next lemma shows that $\mathbf{L}$ behaves like $\slashed{\nabla}$.
\begin{lemma}[equivalence of $\mathbf{L}$ and angular derivative]\label{lem:Gequ}
For each $\ell\in \mathbb{N}_0$, 
\begin{equation}\label{equ:multiL}
 \sum_{j=0}^\ell \Vert \mathbf{L}^j (v,w) \Vert_{G} \approx  \Vert (v,w)\Vert_{H^{\ell+1}\times H^\ell}.
\end{equation}
\end{lemma}
\begin{proof}
We use an induction argument. The assertion for $\ell=0$ follows directly from \eqref{equ:Gequiv}. Now we assume \eqref{equ:multiL} holds for $\ell$ and prove it for $\ell+1$. From the induction hypothesis, 
\begin{align}\label{equ:GappH}
\sum_{j=0}^{\ell+1} \Vert \mathbf{L}^j (v,w) \Vert_{G}\approx \Vert  (v,w)  \Vert_{H^1\times H^0}+\Vert   \mathbf{L}(v,w)  \Vert_{H^\ell\times H^{\ell-1}}
\end{align}
In view of Definition~\ref{def:mathcal},  $ \Vert   \mathbf{L}(v,w)  \Vert_{H^\ell\times H^{\ell-1}}\leq C \Vert (v,w)\Vert_{H^{\ell+1}\times H^\ell}$. Therefore,
\begin{equation}\label{equ:GleqH}
\sum_{j=0}^{\ell+1} \Vert \mathbf{L}^j (v,w) \Vert_{G}\leq C \Vert (v,w)\Vert_{H^{\ell+1}\times H^\ell}.
\end{equation}
To obtain the inequality in the other direction, we use
\begin{align*}
\Vert w \Vert_{H^\ell} &  \leq C    \Vert  2^{-1}m v+
w  \Vert_{H^\ell} +  C\Vert v \Vert_{H^\ell},\\
\Vert v\Vert_{H^{\ell+1}} &   \leq C  \Vert -\mathcal{L}_{\Sigma}v+ 4^{-1}m^2 v+2^{-1} m  w  \Vert_{  H^{\ell-1}}+ C  \Vert  (v,w) \Vert_{H^{\ell}\times H^{\ell-1}}.
\end{align*} 
Combining Definition~\ref{def:mathcal}, \eqref{equ:GappH} and the induction hypothesis, we obtain
\begin{equation}\label{equ:GgeqH}
\Vert (v,w)\Vert_{H^{\ell+1}\times H^\ell}  \leq C   \sum_{j=0}^{\ell+1} \Vert \mathbf{L}^j (v,w) \Vert_{G}.
\end{equation} 
The assertion then follows from \eqref{equ:GleqH} and \eqref{equ:GgeqH}.
\end{proof}
\begin{definition}\label{def:qEkl} 
Fix $k,\ell\in\mathbb{N}_0$. For a section $u\in C^{k+\ell+1}(Q_{0,\infty},\widetilde{\mathbf{V}})$, we define
\begin{equation*}
q^{(k,\ell)}(u):= {\partial_t^k}\mathbf{L}^\ell q(u),\ \mathcal{E}^{(k,\ell)}(u):= {\partial_t^k}\mathbf{L}^\ell \mathcal{E}(u). 
\end{equation*}
Here $q(u)$ and $\mathcal{E}(u)$ are given in Definition~\ref{def:qE}. We often abbreviate them to $q^{(k,\ell)}$, $\mathcal{E}^{(k,\ell)}$ and denote by $q^{(k,\ell)}(t)$, $\mathcal{E}^{(k,\ell)}(t)$ the restriction of $q^{(k,\ell)}(u)$, $\mathcal{E}^{(k,\ell)}(u)$ on $\{t\}\times\Sigma$, respectively.
\end{definition}

\begin{corollary}\label{cor:qE} Fix $s\in\mathbb{N}_0$. Then for all $u\in C^{s+2}(Q_{0,\infty},\widetilde{\mathbf{V}})$ and $t\in (0,\infty)$, 
\begin{equation}\label{equ:qkl}
\sum_{k+\ell \le s } \Vert q^{(k,\ell)}(u) \Vert_{G}(t)  \approx \sum_{k+\ell \le s+1 }\Vert \partial_t^k \slashed{\nabla}^{\ell  }u\Vert_{L^2}(t).
\end{equation}
Moreover, suppose $u$ satisfies $\|u\|_{C^{s+2}}(t)=o(1)$. Then
\begin{equation}\label{equ:Ekl}
\sum_{k+\ell \le s } \Vert \mathcal{E} ^{(k,\ell)}(u)\Vert_{G}(t)  = o(1)\sum_{k+\ell \le s } \Vert q^{(k,\ell)}(u) \Vert_{G}(t)
\end{equation}
\end{corollary}
\begin{proof}
From Definition~\ref{def:qE} and Lemma~\ref{lem:Gequ} (equivalence of $\mathbf{L}$ and angular derivative),
\begin{align*}
\sum_{k+\ell \le s } \Vert q^{(k,\ell)}(u) \Vert_{G}(t)&\approx \sum_{k=0}^s \Vert  (\partial_t^k u,\partial_t^{k+1} u-2^{-1}m \partial_t^k u)  \Vert_{H^{s-k+1}\times H^{s-k} }(t)\\
&\approx \sum_{k+\ell \le s+1 }\Vert \partial_t^k \slashed{\nabla}^{\ell  }u\Vert_{L^2}(t). 
\end{align*}
This gives \eqref{equ:qkl}. Similarly, from Lemma~\ref{lem:Gequ}  and Definition~\ref{def:qE}, 
\begin{align*}
\sum_{k+\ell \le m } \Vert \mathcal{E}^{(k,\ell)}(u) \Vert_{G}(t)&\approx \sum_{k=0}^s \Vert \partial_t^k  E_1(u) \Vert_{ H^{s-k} }(t).  
\end{align*}
From \eqref{equ:Estr} and the assumption, 
\begin{align*}
\sum_{k=0}^s \Vert \partial_t^k  E_1(u) \Vert_{ H^{s-k} }(t)=o(1)\sum_{k+\ell \le s+1 }\Vert \partial_t^k \slashed{\nabla}^{\ell  }u\Vert_{L^2}(t). 
\end{align*}
Then \eqref{equ:Ekl} follows from \eqref{equ:qkl}.
\end{proof}

Let us project the equation \eqref{eq-vectoru} onto vectors in $\mathcal{B}$. Let 
\begin{equation}\label{def:coef_exp}
\begin{split}
&\xi_{i,1}(t):= G(q(t),\psi_{i,1}),\ \xi_{i,2}(t):=G(q(t),\psi_{i,2})\ \textup{for}\ i\in I_1,\\ 
&\xi_{i,3}(t):= G(q(t),\psi_{i,3}),\ \xi_{i,4}(t):=G(q(t),\psi_{i,4})\ \textup{for}\ i\in I_2,\\ 
&\xi^\pm_{i}(t):= G(q(t),\psi^\pm_{i})\ \textup{for}\ i\in I_3\cup I_4.
\end{split}
\end{equation}
Also, let
\begin{equation}\label{def:error_exp}
\begin{split}
&\mathcal{E}_{i,1}(t):= G(\mathcal{E}(t),\psi_{i,1}),\ \mathcal{E}_{i,2}(t):= G(\mathcal{E}(t),\psi_{i,2})\ \textup{for}\ i\in I_1,\\
&\mathcal{E}_{i,3}(t):= G(\mathcal{E}(t),\psi_{i,3}),\ \mathcal{E}_{i,4}(t):= G(\mathcal{E}(t),\psi_{i,4})\ \textup{for}\ i\in I_2,\\
&\mathcal{E}^\pm_{i}(t):= G(\mathcal{E}(t),\psi^\pm_{i})\ \textup{for}\ i\in I_3\cup I_4.
\end{split}
\end{equation}
We can then rewrite \eqref{eq-vectoru} as follows. For $i\in I_1$, 
\begin{equation}\label{equ:odegroup1_exp}
\begin{split}
&\frac{d}{dt} \xi_{i,1}-2^{-1}m \xi_{i,1}+\beta_i\xi_{i,2}=\mathcal{E}_{i,1},\\ 
&\frac{d}{dt} \xi_{i,2}-2^{-1}m \xi_{i,2}-\beta_i\xi_{i,1}=\mathcal{E}_{i,2}. 
\end{split}
\end{equation}
For $i\in I_2$, 
\begin{equation}\label{equ:odegroupJ_exp}
\begin{split}
&\frac{d}{dt} \xi_{i,3}-2^{-1}m \xi_{i,3} =\mathcal{E}_{i,3},\\ 
&\frac{d}{dt} \xi_{i,4}-2^{-1}m \xi_{i,4}- \xi_{i,3}=\mathcal{E}_{i,4}.
\end{split}
\end{equation}
For $i\in I_3\cup I_4$, 
\begin{equation}\label{equ:odegroup_exp}
\frac{d}{dt} \xi^\pm_{i}-\gamma^\pm_i \xi_{i } =\mathcal{E}^\pm_{i }. 
\end{equation}

\section{Fast decaying solutions}\label{sec:exp}
In this section, we consider solutions to \eqref{equ:main} that decay exponentially and prove Theorem~\ref{thm:general_exponential_e}. The proof for Theorem~\ref{thm:general_exponential_p} is simpler so we will omit it. We assume throughout this section that $m<0$ and $I_1, I_2$ are non-empty. The proof can be generalized easily to other cases. We begin with a unique continuation property at infinity. Though we closely follow the argument in \cite{Strehlke}, we include the proof for readers' convenience.  
{\begin{proposition}\label{pro:infinitydecay}
Let $u\in C^\infty(Q_{0,\infty},\widetilde{ \mathbf{V}})$ be a solution to \eqref{equ:main} that satisfies $\|u\|_{C^1}(t)=O(e^{ \gamma t})$ as $t\to \infty$ for all $\gamma<0$. Then $u\equiv 0$.
\end{proposition}
\begin{proof}
By the elliptic regularity, Lemma~\ref{lem:regularity_e},  $\|u\|_{C^s}(t)=O(e^{\gamma t} )$ for all $s\in\mathbb{N}$ and $\gamma <0$. Let $q=q(u)$ be given in Definition~\ref{def:qE}. Take $\gamma<2^{-1}m-1$. Let $\Pi_\gamma q$ be the projection of $q$ onto the eigenspace of $\mathbf{L}$ whose eigenvalues are less than or equal to  $\gamma $. Namely, in terms of the coefficients introduced in \eqref{def:coef_exp}, 
\begin{align*}
(\Pi_\gamma q)(t)=\sum_{i:\gamma^-_i \le  \gamma} \xi_i^-(t)\psi_i^-.
\end{align*}

\bigskip 

We claim that for any $t_2 \ge t_1 $, 
\bea \label{eq-mainprop41} e^{-\gamma (t_2 -t_1 )} \Vert  q(t_2) \Vert_G \le \Vert\Pi_\gamma   q(t_1 )\Vert_G + \int_{t_1 }^\infty e^{ -\gamma (s-t_1 )} \Vert \mathcal{E}(s)\Vert_G \, ds.  \eea Suppose \eqref{eq-mainprop41} is true at the moment. From \eqref{equ:Estr}, there exists a uniform constant $C<\infty$ such that if $\Vert u\Vert_{C^1}(t)\le 1$, then there holds
\begin{align*}
\Vert \mathcal{E}(t)\Vert _{G } \le C  \Vert q(t)\Vert _{G}\Vert u\Vert _{C^2}(t).  
\end{align*}
Define $M_\gamma(t)=\sup_{\tau \geq t} e^{-\gamma(\tau-t)} \Vert q(\tau) \Vert_G $. Suppose $t_1$ is large enough such that $\Vert u\Vert _{C^1}(s )\le 1$ for $s\ge t_1$. Then from \eqref{eq-mainprop41}, 
\begin{align*}
M_\gamma(t_1) \le \Vert \Pi_\gamma q(t_1)\Vert_G + C  \int_{t_1}^\infty  \Vert u\Vert_{C^2}(s) \, ds \cdot M_\gamma(t_1).
\end{align*}
From the exponential decay assumption we may choose large $t_1$ so that $C \int _{t_1}^\infty \Vert u\Vert_{C^2}(s) ds\le 1/2$. Hence
\begin{align}\label{eq-4777}
M_\gamma(t_1)  \le 2\Vert \Pi_\gamma  q(t_1)\Vert_G.
\end{align}
It is clear that $M_\gamma(t_1)$ is non-increasing in $\gamma$. Therefore,  
\begin{align*}
M_\gamma(t_1)\leq \limsup_{\gamma'\to -\infty } M_{\gamma'}(t_1)\leq 2\limsup_{\gamma'\to -\infty }  \Vert \Pi_{\gamma'}  q(t_1)\Vert_G=0.
\end{align*} 
We conclude $q(t)=0$ (and thus $u(t)=0$) for $t\ge t_1$. Next, we show $q(t)=0$ upto $t=0$. Suppose on the contrary $t_0=\inf \{ t_1 \, :\, q(t)=0 \text{ for } t\ge t_1\}>0$. By the smoothness of $u(x,t)$, we may find a small $\e>0$ such that $\Vert u\Vert_{C^1}(t) \le 1$ for $t\in [t_0-\e, t_0]$ and 
\[C\int_{t_0-\e}^\infty \Vert u\Vert _{C^2}(s)\, ds =C\int_{t_0-\e}^{t_0}  \Vert u\Vert _{C^2}(s)\, ds \le \frac12 ,  \]
which implies that \eqref{eq-4777} holds for $t_1=t_0-\e$. This gives a contradiction and proves the statement.

\bigskip

It remains to show \eqref{eq-mainprop41}. Define non-negative functions $X_\pm(t)$ by
\begin{align*}
X^2_+(t)= &\sum_{i\in I_1} |\xi_{i,1}(t)|^2+|\xi_{i,2}(t)|^2+\sum_{i\in I_2} |\xi_{i,3}(t)|^2+|\xi_{i,4}(t)|^2 + \sum |\xi^+_i(t)|^2+\sum_{i:\gamma^-_i>\gamma} |\xi^-_i(t)|^2,\\
X^2_-(t)= &\sum_{i:\gamma^-_i\leq \gamma} |\xi^-_i(t)|^2.
\end{align*}
From \eqref{equ:odegroup1_exp}-\eqref{equ:odegroup_exp}, $X_+(t)X_+'(t)$ equals,
\begin{align*}
 &2^{-1}m\left(\sum_{i\in I_1} |\xi_{i,1}(t)|^2+|\xi_{i,2}(t)|^2+\sum_{i\in I_2} |\xi_{i,3}(t)|^2+|\xi_{i,4}(t)|^2 \right)+ \sum \gamma^+_i |\xi^+_i(t)|^2+\sum_{i:\gamma^-_i>\gamma} \gamma^-_i |\xi^-_i(t)|^2\\
 &+\sum_{i\in I_2}\xi_{i,3}(t)\xi_{i,4}(t)+\sum \xi^+_i(t)\mathcal{E}^+_i(t)+\sum_{i:\gamma^-_i>\gamma} \xi^+_-(t)\mathcal{E}^+_-(t)+\sum_{i\in I_1}  \xi_{i,1}(t)\mathcal{E}_{i,1}(t) +\xi_{i,1}(t)\mathcal{E}_{i,1}(t)\\
&+\sum_{i\in I_2}  \xi_{i,3}(t)\mathcal{E}_{i,3}(t) +\xi_{i,4}(t)\mathcal{E}_{i,4}(t).
\end{align*}
From $\left| \sum_{i\in I_2}\xi_{i,3}(t)\xi_{i,4}(t) \right|\leq  2^{-1}\sum_{i\in I_2}|\xi_{i,3}(t)|^2+|\xi_{i,4}(t)|^2$ and $\gamma<2^{-1}m-1$, 
\begin{align}\label{equ:42}
X_+'(t)\geq \gamma X_+'(t)+Y_+(t).
\end{align}
Here $Y_+(t)$ is given by
\begin{align*}
X_+(t)Y_+(t)=&\sum \xi^+_i(t)\mathcal{E}^+_i(t)+\sum_{i:\gamma^-_i>\gamma} \xi^+_-(t)\mathcal{E}^+_-(t)+\sum_{i\in I_1}  \xi_{i,1}(t)\mathcal{E}_{i,1}(t) +\xi_{i,1}(t)\mathcal{E}_{i,1}(t)\\
&+\sum_{i\in I_2}  \xi_{i,3}(t)\mathcal{E}_{i,3}(t) +\xi_{i,4}(t)\mathcal{E}_{i,4}(t).
\end{align*}
Similarly, we have
\begin{align}\label{equ:43}
X_-'(t)\leq \gamma X_-'(t)+Y_-(t),
\end{align}
where $Y_-(t)$ is given by $X_-(t)Y_-(t)= \sum_{i:\gamma^-_i \le \gamma} \xi^-_i(t)\mathcal{E}^-_i(t)$. Fix $t_2\geq t_1$. By integrating \eqref{equ:43} from $t_1$ to $t_2$, we obtain
\begin{equation}\label{equ:44}
\begin{split}
e^{-\gamma (t_2-t_1)}X_-(t_2)\leq   &X_-(t_1)+\int_{t_1}^{t_2} e^{-\gamma (s-t_1)} Y_-(s)\, ds \\
\leq & X_-(t_1)+\int_{t_1}^{\infty} e^{-\gamma (s-t_1)} |Y_-(s)|\, ds.
\end{split}
\end{equation}
Take any $t_3\geq t_2$. By integrating \eqref{equ:42} from $t_2$ to $t_3$, we obtain
\begin{align*}
e^{-\gamma (t_2-t_1)}X_+(t_2)\leq   &e^{-\gamma (t_3-t_1)}X_+(t_3)-\int_{t_2}^{t_3} e^{-\gamma (s-t_1)} Y_+(s)\, ds \\
\leq &e^{-\gamma (t_3-t_1)}X_+(t_3)+\int_{t_1}^{\infty} e^{-\gamma (s-t_1)} |Y_+(s)|\, ds 
\end{align*} 
By the decay assumption, $e^{-\gamma (t_3-t_1)}X_+(t_3)$ goes to zero when $t_3$ goes to infinity. Hence 
\begin{align}\label{equ:45}
e^{-\gamma (t_2-t_1)}X_+(t_2)\leq  \int_{t_1}^{\infty} e^{-\gamma (s-t_1)} |Y_+(s)|\, ds.
\end{align} 
Note that $X_+^2(t)+X_-^2(t)=\Vert q(t) \Vert^2_G $, $X_-(t)=\Vert \Pi_\gamma q(t) \Vert_G$ and $ Y_+^2(t)+Y_-^2(t)\leq  \Vert \mathcal{E}(t) \Vert^2_G $. Then \eqref{eq-mainprop41} follows from \eqref{equ:44} and \eqref{equ:45}.
\end{proof}

\begin{remark}
A similar argument applies for the parabolic equation. The only difference is that to control the error term $E_2(u)$ in \eqref{equ:linear_p}, one needs to differentiate the equation. See Lemma~\ref{lem:error_p}.
\end{remark}
} 
Let $u$ be an exponentially decaying solution to \eqref{equ:main}. Namely, $\|u\|_{C^1}(t)=O(e^{-2\varepsilon_0 t})$ for some $\varepsilon_0>0$. We further assume $u$ is not identically zero. In view of Proposition~\ref{pro:infinitydecay}, the set
\begin{align*}
\Lambda:=\{\gamma<0\, :\, \|u\|_{C^1}(t)=O(e^{ \gamma t})\}
\end{align*}
has an infimum $-\infty<\gamma_*<0$. From the elliptic regularity, Lemma~\ref{lem:regularity_e}, $\|u\|_{C^s}(t)=O(e^{(\gamma_*+\varepsilon) t})$ for all $s\in\mathbb{N}$ and $\varepsilon>0$. Let $\{\xi_{i,1}(t), \xi_{i,2}(t)\}_{i\in I_1}$, $\{\xi_{i,3}(t), \xi_{i,4}(t)\}_{i\in I_2}$, $\{\xi^\pm_i(t)\}_{i\in I_3\cup I_4 }$ be the coefficients defined in \eqref{def:coef_exp}. Let $\{\mathcal{E}_{i,1}(t), \mathcal{E}_{i,2}(t)\}_{i\in I_1}$, $\{\mathcal{E}_{i,3}(t), \mathcal{E}_{i,4}(t)\}_{i\in I_2}$, $\{\mathcal{E}^\pm_i(t)\}_{i\in I_3\cup I_4 }$ be given by \eqref{def:error_exp}. The quadratic nature of $E_1(u)$ (see \eqref{equ:Estr}), in particular, implies  
\begin{equation}\label{equ:error_exp}
\begin{split} \Vert(0, E_1(u))\Vert^2_{G}(t) =
\sum_{i\in I_1} (\mathcal{E}_{i,1}(t))^2+(\mathcal{E}_{i,2}(t))^2+\sum_{i\in I_2} (\mathcal{E}_{i,3}(t))^2+(\mathcal{E}_{i,4}(t))^2+\sum_{i\in I_3\cup I_4} (\mathcal{E}_{i}^+(t))^2+(\mathcal{E}_{i}^-(t))^2\\
=O(e^{2(\gamma_*-\varepsilon_0)t}).
\end{split}
\end{equation} Here we used $|E_1(u)| \le C e^{-2\varepsilon_0t} e^{(\gamma^*+\varepsilon_0)t} = Ce^{(\gamma^*-\varepsilon_0)t}$.

\begin{lemma}\label{lem:gamma_star}
There holds $\gamma_*\in \{\gamma^+_i,\gamma^-_i\}_{i\in I_3\cup I_4}\cup\{2^{-1}m\}$.
\end{lemma}
\begin{proof}
Suppose the assertion fails. This implies there exists $\varepsilon_1\in (0,\varepsilon_0)$ such that there is no element of $\{\gamma^+_i,\gamma^-_i\}_{i\in I_3\cup I_4}\cup\{2^{-1}m\}$ in the interval $[\gamma_*-2\varepsilon_1,\gamma_*+2\varepsilon_1]$. We show this leads to a contradiction for the case $2^{-1}m< \gamma_*$. The argument for the case $\gamma_*<2^{-1}m$ is similar. Define an non-negative function $X_+(t)$ by
\begin{align*}
 X^2_+(t) =\sum_{i:\gamma^+_i>\gamma_*} |\xi_i^+(t)|^2+\sum_{i:\gamma^-_i>\gamma_*} |\xi_i^-(t)|^2.
\end{align*}
From \eqref{equ:odegroup_exp},
\begin{align*}
 X_+(t)X_+'(t) =&\sum_{i:\gamma^+_i>\gamma_*} \gamma^+_i|\xi_i^+(t)|^2+\sum_{i:\gamma^-_i>\gamma_*} \gamma^-_i|\xi_i^-(t)|^2 + \sum_{i:\gamma^+_i>\gamma_*}  \xi_i^+(t)\mathcal{E}^+_i(t) +\sum_{i:\gamma^-_i>\gamma_*} \xi_i^-(t)\mathcal{E}^-_i(t)\\
 \geq &(\gamma_*+\varepsilon_1) X^2_+(t)  +X_+(t)Y_+(t).
\end{align*}
Here $Y_+(t)$ is given by 
\begin{align*}
X_+(t)Y_+(t) = \sum_{i:\gamma^+_i>\gamma_*}  \xi_i^+(t)\mathcal{E}^+_i(t) +\sum_{i:\gamma^-_i>\gamma_*} \xi_i^-(t)\mathcal{E}^-_i(t).
\end{align*}
From the Cauchy-Schwarz inequality and \eqref{equ:error_exp}, $|Y_+(t)|=O(e^{(\gamma_*-\varepsilon_0)t})$. Using $$\lim_{t\to\infty} e^{-(\gamma_*+\varepsilon_1)t}X_+(t)=0,$$  
we can integrate $$\frac{d}{dt}\left(e^{-(\gamma_*+\varepsilon_1)t}X_+(t)  \right) \geq  e^{-(\gamma_*+\varepsilon_1)t}Y_+(t)$$ 
from $t$ to $\infty$ to obtain 
\begin{align}\label{equ:X+}
X_+(t)\leq e^{(\gamma_*+\varepsilon_1)t}\int_t^\infty e^{-(\gamma_*+\varepsilon_1)\tau }|Y_+(\tau)|\, d\tau=O(e^{(\gamma_*-\varepsilon_0)t}).
\end{align}
Let
\begin{align*}
 X^2_-(t) =\sum_{i:\gamma^+_i<\gamma_*} |\xi_i^+(t)|^2+\sum_{i:\gamma^-_i<\gamma_*} |\xi_i^-(t)|^2+\sum_{i\in I_1} |\xi_{i,1}(t)|^2+|\xi_{i,2}(t)|^2+\sum_{i\in I_2} |\xi_{i,3}(t)|^2+\varepsilon_1^2 |\xi_{i,4}(t)|^2.
\end{align*}
From \eqref{equ:odegroup1_exp}-\eqref{equ:odegroup_exp}, $ X_-(t)X_-'(t)$ equals
\begin{align*}
&\sum_{i:\gamma^+_i<\gamma_*} \gamma^+_i|\xi_i^+(t)|^2+\sum_{i:\gamma^-_i<\gamma_*} \gamma^-_i|\xi_i^-(t)|^2+2^{-1}m\left(\sum_{i\in I_1}|\xi_{i,1}(t)|^2+|\xi_{i,2}(t)|^2+\sum_{i\in I_2} |\xi_{i,3}(t)|^2+\varepsilon_1^2 |\xi_{i,4}(t)|^2\right) \\
 +&\varepsilon_1^2 \sum_{i\in I_2}\xi_{i,3}(t)\xi_{i,4}(t)+ \sum_{i:\gamma^+_i<\gamma_*}  \xi_i^+(t)\mathcal{E}^+_i(t) +\sum_{i:\gamma^-_i<\gamma_*} \xi_i^-(t)\mathcal{E}^-_i(t)+\sum_{i\in I_1} \xi_{i,1}(t)\mathcal{E}_{i,1}(t) + \xi_{i,2}(t)\mathcal{E}_{i,2}(t)\\
+&\sum_{i\in I_2} \xi_{i,3}(t)\mathcal{E}_{i,3}(t) + \varepsilon_1^2 \xi_{i,4}(t)\mathcal{E}_{i,4}(t).
\end{align*}
From $\varepsilon_1^2 \sum_{i\in I_2}|\xi_{i,3}(t)\xi_{i,4}(t)|\leq 2^{-1}\varepsilon_1  \sum_{i\in I_2}|\xi_{i,3}(t)|^2+\varepsilon^2_1|\xi_{i,4}(t)|^2$, \begin{align*}
X_-(t)X_-'(t)\leq (\gamma_*-\varepsilon_1) X^2_-(t)  +X_-(t)Y_-(t).
\end{align*}
Here $Y_-(t)$ is given by
\begin{align*}
X_-(t)Y_-(t)=&\sum_{i:\gamma^+_i<\gamma_*}  \xi_i^+(t)\mathcal{E}^+_i(t) +\sum_{i:\gamma^-_i<\gamma_*} \xi_i^-(t)\mathcal{E}^-_i(t)+\sum_{i\in I_1} \xi_{i,1}(t)\mathcal{E}_{i,1}(t) + \xi_{i,2}(t)\mathcal{E}_{i,2}(t)\\ 
&+ \sum_{i\in I_2} \xi_{i,3}(t)\mathcal{E}_{i,3}(t) + \varepsilon_1^2 \xi_{i,4}(t)\mathcal{E}_{i,4}(t).
\end{align*} 
From the Cauchy-Schwarz inequality and \eqref{equ:error_exp},  $|Y_-(t)|=O(e^{(\gamma_*-\varepsilon_0)t})$. Integrating 
$$\frac{d}{dt}\left(e^{-(\gamma_*-\varepsilon_1)t}X_-(t)  \right) \geq  e^{-(\gamma_*-\varepsilon_1)t}Y_-(t)$$ 
from $0$ to $t$, we obtain 
\begin{align}\label{equ:X-}
X_-(t)\leq e^{(\gamma_*-\varepsilon_1)t}\left( X_-(0)+ \int_0^\infty e^{-(\gamma_*-\varepsilon_1)\tau} |Y_-(\tau)| d\tau\right)=O(e^{(\gamma_*-\varepsilon_1)t}).
\end{align}

Combining \eqref{equ:X+} and \eqref{equ:X-},  $\Vert q(t) \Vert_G=O(e^{(\gamma_*-\varepsilon_1)t})$. From the elliptic regularity, Lemma~\ref{lem:regularity_e}, we then have $\Vert u \Vert_{C^1}(t)=O(e^{(\gamma_*-\varepsilon_1)t})$. This contradicts to the definition of $\gamma_*$. 
\end{proof}

\begin{lemma}\label{lem:gamma_star_m}
Suppose $\gamma_*=2^{-1}m$. Then there exists $w_i \in\mathbb{C} $ for $i\in I_1$ and $c_{i,3}, c_{i,4}\in\mathbb{R} $ for $i\in I_2$ such that the following holds. For
\begin{align*}
\hat{q}(t):= q(t) - e^{\gamma_* t} \bigg(\sum_{i\in I_1 } \textup{Re}\big( w_i e^{\mathbf{i}\beta_i t}  \big) \psi_{i,1}+\textup{Im}\big( w_i e^{\mathbf{i}\beta_i t}  \big) \psi_{i,2} + \sum_{i\in I_2 } c_{i,3}\psi_{i,3}+(tc_{i,3}+c_{i,4})\psi_{i,4}\bigg)    
\end{align*}
there exists $\varepsilon >0$ such that $\Vert \hat{q}(t) \Vert_G= O(e^{(\gamma_*-\varepsilon)t})$.
\end{lemma}
\begin{proof}
Define non-negative functions $X_\pm(t)$ by
\begin{align*}
X^2_+(t)&=\sum_{i:\gamma^+_i>\gamma_*} (\xi_i^+(t))^2+\sum_{i:\gamma^-_i>\gamma_*} (\xi_i^-(t))^2,\\
 X^2_-(t)&=\sum_{i:\gamma^+_i<\gamma_*} (\xi_i^+(t))^2+\sum_{i:\gamma^-_i<\gamma_*} (\xi_i^-(t))^2.
\end{align*}
By shrinking the value of $\varepsilon_0$ if necessary, we may assume $\varepsilon_0<|2^{-1}m-\gamma^\pm_i|$ for all $i\in I_3\cup I_4$. An argument similar to the one in the proof of Lemma~\ref{lem:gamma_star} shows $X_\pm(t)=O(e^{(\gamma_*-\varepsilon_0)t})$. From \eqref{equ:odegroup1_exp}, we derive for $i\in I_1$
\begin{align*}
\frac{d}{dt}\left( e^{-(\gamma_* +\mathbf{i}\beta_i)t} \left(\xi_{i,1}(t)+\mathbf{i}\xi_{i,2}(t) \right)\right)=  e^{-(\gamma_* +\mathbf{i}\beta_i)t} \left(\mathcal{E}_{i,1}(t)+\mathbf{i}\mathcal{E}_{i,2}(t) \right) .
\end{align*}
Integrating the above from $0$ to $t$ and using \eqref{equ:error_exp} yield 
\begin{align*}
\xi_{i,1}(t)+\mathbf{i}\xi_{i,2}(t)= w_i e^{(\gamma_*+\mathbf{i}\beta_i)t}+O(e^{(\gamma_*-\varepsilon_0)t})
\end{align*}
for some $w_i\in\mathbb{C}$. A similar argument applying to $\frac{d}{dt}\left( e^{-\gamma_* t} \xi_{i,3}(t)\right)= e^{-\gamma_* t} \mathcal{E}_{i,3}(t)$ gives 
$$\xi_{i,3}(t)= c_{i,3} e^{\gamma_* t}+O(e^{(\gamma_*-\varepsilon_0)t})$$ 
for some $c_{i,3}\in\mathbb{R}$. Lastly, from  
\begin{align*}
\frac{d}{dt}\left( e^{-\gamma_* t} \xi_{i,4}(t)\right) =& e^{-\gamma_* t} \xi_{3,i}(t) +e^{-\gamma_* t} \mathcal{E}_{i,4}(t)=  c_{i,3}+O(e^{  -\varepsilon_0t}), 
\end{align*}
we derive $\xi_{4,i}=(tc_{i,3}+c_{i,4})e^{\gamma_*t}+O(e^{(\gamma_*-\varepsilon_0)t})$ for some $c_{i,4}\in\mathbb{R}$.
\end{proof}

For $\gamma_*=\gamma^+_i$ or $\gamma^-_i$, an analogous result holds. We omit the proof because it is similar to and simpler than the one for Lemma~\ref{lem:gamma_star_m}.  
\begin{lemma}\label{lem:gamma_star_pm}
Suppose $\gamma_*=\gamma^+_i$ for some $i\in I_3\cup I_4$ and let $N$ be the multiplicity of $\lambda_i$. We may assume $\lambda_{i}=\lambda_{i+1}=\dots =\lambda_{i+N-1}$. Then there exists $a_1,a_2,\dots, a_N\in\mathbb{R}$ and $\varepsilon>0$ such that
\begin{align*}
\Vert q(t)- e^{\gamma_* t} \sum_{j=1}^N a_j \psi^+_{i+j-1} \Vert_G=O(e^{(\gamma_*-\varepsilon) t}).
\end{align*}
Similar result holds when $\gamma_*=\gamma^-_i$ for some $i\in I_3\cup I_4$ 
\end{lemma}
We are ready to prove Theorem~\ref{thm:general_exponential_e}.
\begin{proof}[Proof of Theorem~\ref{thm:general_exponential_e}]
Suppose $\gamma_*=\gamma^+_i$ or $\gamma^-_i$. From Lemma~\ref{lem:gamma_star_pm} and Corollary~\ref{cor:qE}, there exists $v$ with $\mathcal{L}_{\Sigma} v=\lambda_i v$ and $\varepsilon>0$ such that $
\Vert u(t)- e^{\gamma_* t}v \Vert_{H^1}= O(e^{(\gamma_*-\varepsilon ) t})$. Let $w=v/\Vert v\Vert_{L^2}$. Then
\begin{align*}
\lim_{t\to\infty} u(t)/\Vert u(t) \Vert_{L^2}=w\ \textup{in}\ H^1(\Sigma,\mathbf{V}).
\end{align*}
We may upgrade the convergence to $C^\infty(\Sigma,\mathbf{V})$ by taking $\mathbf{L}$ derivatives to \eqref{eq-vectoru}. As a result, case (1) in Theorem~\ref{thm:general_exponential_e} holds.

Suppose $\gamma_*=2^{-1}m$. From Lemma~\ref{lem:gamma_star_m} and Corollary~\ref{cor:qE}, there exists $w_i\in\mathbb{C}$ for $i\in I_1$ and $c_{i,3}, c_{i,4}\in I_2$ such that
\begin{align*}
 \left\Vert  u(t) - e^{\gamma_*t}\sum_{i\in I_1}\textup{Re} \left(w_ie^{\mathbf{i}\beta_i t}  \right)\varphi_i- e^{\gamma_*t}\sum_{i\in I_2}(tc_{i,3}+c_{i,4})  \varphi_i\right\Vert_{H^1} = O(e^{(\gamma_*-\varepsilon ) t}). 
\end{align*}
Then case (2) or case (3) in Theorem~\ref{thm:general_exponential_e} holds depending on whether $c_{i,3}\equiv 0$ or not.
\end{proof}

\section{Slowly decaying solutions to elliptic equation}\label{sec:elliptic}

In this section, we show that if a solution to \eqref{equ:main} decays slowly, then the neutral mode, the projection of $q(u)$ onto the $0$-eigenspace of $\mathbf{L}$, dominates the solution. Moreover, the neutral mode evolves by a gradient flow up to a small error. That is the content of Proposition~\ref{prop-neutral-dynamics}. 
\vspace{0.1cm}

Let $u\in C^\infty(Q_{0,\infty},\widetilde{ \mathbf{V}})$ be a solution to \eqref{equ:main} with $\|u\|_{C^1}(t)=o(1)$ as $t\to \infty$. From the elliptic regularity, Lemma~\ref{lem:regularity_e}, $\|u\|_{C^s}(t)=o(1)$ for all $s\in\mathbb{N}$. We further assume that $u$ does not decay exponentially. Namely, for any $\varepsilon>0$, 
\begin{equation}\label{equ:nonexpdecay}
\limsup_{t\to\infty} e^{\varepsilon t}\|u\|_{C^1}(t)=\infty.
\end{equation}

Recall that we rewrote \eqref{equ:main} as an ODE system \eqref{equ:odegroup1_exp}-\eqref{equ:odegroup_exp}. For brevity, we assume throughout this section that $I_2=\varnothing$. With notational changes, the proof can be readily extended to cover the case where $I_2\neq \varnothing$. Since $I_2=\varnothing,$ the ODE system consists of \eqref{equ:odegroup1_exp} and \eqref{equ:odegroup_exp}. It is convenient to relabel the coefficients $\{\xi^\pm_i\}_{i\in I_3\cup I_4}$ in \eqref{equ:odegroup_exp}. For $\{\psi^\pm_i\}_{i\in I_4}$, we set 
\begin{align*}
& \{\Psi_i\}_{i\in\mathbb{N}}= \{\psi^+_i\, |\, i\in I_4\ \textup{and}\ \gamma^{+}_i>0\}\cup \{\psi^-_i\, |\, i\in I_4\ \textup{and}\ \gamma^{-}_i>0\},\\
 &  \{\Psi_i\}_{i\in-\mathbb{N}}=\{\psi^+_i\, |\, i\in I_4\ \textup{and}\ \gamma^{+}_i<0\}\cup \{\psi^-_i\, |\, i\in I_4\ \textup{and}\ \gamma^{-}_i<0\},
\end{align*}
and define $\Gamma_i$, a relabelling of $\gamma^\pm_i $ for $i\in I_4$, by
\begin{align}\label{equ:Psi}
\mathbf{L}\Psi_i=\mathbf{L}^\dagger\Psi_i=\Gamma_i\Psi_i. 
\end{align}
Recall that $I_3=\{\iota+1,\dots, \iota+J\}$. If $m>0$, we set for $1\leq j\leq J$ $ \Upsilon_j =\psi^-_{\iota+j} $ and $ \overline{\Upsilon}_j = \psi_{\iota+j}^+$. If $m<0$, we set $ \Upsilon_j =\psi^+_{\iota+j} $ and $ \overline{\Upsilon}_j = \psi_{\iota+j}^-$. This arrangement ensures that
\begin{equation}\label{equ:Ups}
\begin{split}
\mathbf{L}\Upsilon_j=\mathbf{L}^\dagger\Upsilon_j=0,\ \mathbf{L}\overline{\Upsilon}_j=\mathbf{L}^\dagger\overline{\Upsilon}_j=m\overline{\Upsilon}_j.  
\end{split}
\end{equation}
Let
\begin{equation}\label{def:coef}
\begin{split}
&\xi_{i}(t):= G(q(t),\Psi_{i})\ \textup{for}\ i\in \mathbb{Z}\setminus\{0\},\\
&z_j(t):= G(q(t),\Upsilon_j),\ \bar{z}_j(t):=G(q(t),\overline{\Upsilon}_j) \ \textup{for}\ 1\leq j\leq J,
\end{split}
\end{equation}
and
\begin{equation}\label{def:error}
\begin{split}
&\mathcal{E}_{i}(t):= G(\mathcal{E}(t),\Psi_{i})\ \textup{for}\ i\in \mathbb{Z}\setminus\{0\},\\
&\mathcal{W}_j(t):= G(\mathcal{E}(t),\Upsilon_j), \ \overline{\mathcal{W}}_j(t):=G(\mathcal{E}(t),\overline{\Upsilon}_j) \ \textup{for}\ 1\leq j\leq J.
\end{split}
\end{equation}
We rewrite \eqref{equ:odegroup_exp} as \eqref{equ:odegroup2} and \eqref{equ:odegroup3} below. For $i\in\mathbb{Z}\setminus\{0\}$,
\begin{equation}\label{equ:odegroup2}
\frac{d}{dt} \xi_{i}-\Gamma_i \xi_{i } =\mathcal{E}_{i }, 
\end{equation}
and for $1\leq j\leq J$,
\begin{equation}\label{equ:odegroup3}
\begin{split}
&\frac{d}{dt} z_j =\mathcal{W}_j,\ \frac{d}{dt} \bar{z}_j-m \bar{z}_j =\overline{\mathcal{W}}_j.
\end{split}
\end{equation} 

We denote $z(t):=(z_1(t),\dots z_{J}(t))$ and $\bar{z}(t):=(\bar{z}_1(t),\dots \bar{z}_{J}(t))$, and use $|z(t)|$ and $|\bar{z}(t)|$ to denote their Euclidean norms respectively. Recall the reduced functional $f$ is introduced in Proposition~\ref{pro:AS}. The goal of this section is to prove the following proposition.

\begin{proposition}\label{prop-neutral-dynamics}  
For a slowly decaying solution $u$ to \eqref{equ:main} as described above, there holds
\begin{align}\label{equ:zdominates}
   |\bar{z}(t)|^2+\sum_{i\in I_1}\left(|\xi_{i,1}(t)|^2+|\xi_{i,2}(t)|^2 \right)+ \sum_{i\neq 0}|\xi_i(t)|^2  = o(1) |z(t)|^2.
\end{align}
Moreover, for all $\varepsilon>0$ there exists a positive constant $C=C(u,m,\mathcal{M}_\Sigma, N_1, \varepsilon)$ such that 
\begin{align}\label{equ:zgradient}
 |z'(t)+m^{-1}\nabla f(z(t))| \leq C   |z(t)|^{p-\varepsilon/2}.
\end{align}
\end{proposition} 

The remainder of this section is dedicated to proving Proposition~\ref{prop-neutral-dynamics}, which involves two main parts. In Subsection~\ref{sec:5.1}, we demonstrate (as shown in Corollary~\ref{cor:sobolev}) that any $C^s$ norm of $u$ can be bounded by $|z(t)|$. In Subsection~\ref{sec:5.2}, we obtain an enhanced decay rate in Lemma~\ref{lem-iterate2} through the decomposition \eqref{def:ut}. Proposition~\ref{prop-neutral-dynamics} is then a simple consequence of Lemma~\ref{lem-iterate2}. 

\subsection{Bounding $\Vert u \Vert_{C^s}$}\label{sec:5.1} To estimate $C^s$ norms of $u$, we need higher-derivative versions of the ODE system. Fix $k,\ell\in\mathbb{N}_0$. Recall that $q^{(k,\ell)}$ and $\mathcal{E}^{(k,\ell)}$ are given in Definition~\ref{def:qEkl}. Let 
\begin{equation}\label{def:coefkl} 
\begin{split}
&\xi^{(k,\ell)}_{i,1}(t):= G(q^{(k,\ell)}(t),\psi_{i,1}),\ \xi^{(k,\ell)}_{i,2}(t):=G(q^{(k,\ell)}(t),\psi_{i,2})\ \textup{for}\ i\in I_1,\\ 
&\xi^{(k,\ell)}_{i}(t):= G(q^{(k,\ell)}(t),\Psi_{i})\ \textup{for}\ i\in \mathbb{Z}\setminus\{0\},\\
&z^{(k,\ell)}_j(t):= G(q^{(k,\ell)}(t),\Upsilon_j),\ \bar{z}^{(k,\ell)}_j(t):=G(q^{(k,\ell)}(t),\overline{\Upsilon}_j) \ \textup{for}\ 1\leq j\leq J.
\end{split}
\end{equation}
Also, let
\begin{equation}\label{def:errorkl} 
\begin{split}
&\mathcal{E}^{(k,\ell)}_{i,1}(t):= G(\mathcal{E}^{(k,\ell)}(t),\psi_{i,1}),\ \mathcal{E}^{(k,\ell)}_{i,2}(t):= G(\mathcal{E}^{(k,\ell)}(t),\psi_{i,2})\ \textup{for}\ i\in I_1,\\
&\mathcal{E}^{(k,\ell)}_{i}(t):= G(\mathcal{E}^{(k,\ell)}(t),\Psi_{i})\ \textup{for}\ i\in \mathbb{Z}\setminus\{0\},\\
&\mathcal{W}^{(k,\ell)}_j(t):= G(\mathcal{E}^{(k,\ell)}(t),\Upsilon_j), \ \overline{\mathcal{W}}^{(k,\ell)}_j(t):=G(\mathcal{E}^{(k,\ell)}(t),\overline{\Upsilon}_j) \ \textup{for}\ 1\leq j\leq J.
\end{split}
\end{equation}
Then for $i\in I_1$,
\begin{equation}\label{equ:odegroup1kl} 
\begin{split}
&\frac{d}{dt} \xi^{(k,\ell)}_{i,1}-2^{-1}m \xi^{(k,\ell)}_{i,1}+\beta_i\xi^{(k,\ell)}_{i,2}=\mathcal{E}^{(k,\ell)}_{i,1},\\ 
&\frac{d}{dt} \xi^{(k,\ell)}_{i,2}-2^{-1}m \xi^{(k,\ell)}_{i,2}-\beta_i\xi^{(k,\ell)}_{i,1}=\mathcal{E}^{(k,\ell)}_{i,2}, 
\end{split}
\end{equation}
for $i\in\mathbb{Z}\setminus\{0\}$,
\begin{equation}\label{equ:odegroup2kl} 
\frac{d}{dt} \xi^{(k,\ell)}_{i}-\Gamma_i \xi^{(k,\ell)}_{i } =\mathcal{E}^{(k,\ell)}_{i }, 
\end{equation}
and for $1\leq j\leq J$,
\begin{equation}\label{equ:odegroup3kl} 
\begin{split}
&\frac{d}{dt} z^{(k,\ell)}_j =\mathcal{W}^{(k,\ell)}_j,\ \frac{d}{dt} \bar{z}^{(k,\ell)}_j-m \bar{z}^{(k,\ell)}_j =\overline{\mathcal{W}}^{(k,\ell)}_j.
\end{split}
\end{equation}

In the next lemma, we use the Merle-Zaag ODE lemma \cite{MZ} (see Lemma \ref{lem-MZODE}) to show that $|z(t)|$ dominates the other coefficients, thereby obtaining a stronger version of  \eqref{equ:zdominates}.

\begin{lemma}[dominance of neutral mode]\label{lem:ODE}
For any $s\in\mathbb{N}_0$,
\begin{align*}
\sum_{k+\ell\leq s}\left[ |z^{(k,\ell)}(t)|^2+|\bar{z}^{(k,\ell)}(t)|^2+\sum_{i\in I_1}\left(|\xi^{(k,\ell)}_{i,1}(t)|^2+|\xi^{(k,\ell)}_{i,2}(t)|^2 \right)+\right.&\left.\sum_{i\neq 0}|\xi^{(k,\ell)}_i(t)|^2 \right] \\
&=(1+o(1))|z(t)|^2.
\end{align*}
\end{lemma}
 
\begin{proof} Fix $s\in\mathbb{N}_0$. We give the proof for the case $m>0$. The argument for $m<0$ is similar. Define three non-negative functions $X_+(t)$, $X_0(t)$ and $X_-(t)$ by
\begin{align*}
X^2_0(t)=&\sum_{k+\ell\leq s}\sum_{1\leq j\leq J}|z^{(k,\ell)}_j(t)|^2,\ X^2_-(t)=\sum_{k+\ell\leq s}\sum_{i\in -\mathbb{N}}|\xi^{(k,\ell)}_i(t)|^2, 
\end{align*}
and
\begin{align*}
X^2_+(t)=\sum_{k+\ell\leq s}\left(\sum_{i\in I_1}\left( |\xi^{(k,\ell)}_{i,1}(t)|^2+|\xi^{(k,\ell)}_{i,2}(t)|^2 \right)+\sum_{i\in\mathbb{N}}|\xi^{(k,\ell)}_i(t)|^2+\sum_{1\leq j\leq J}|\bar{z}^{(k,\ell)}_j(t)|^2 \right).
\end{align*}
From \eqref{equ:odegroup1kl}, \eqref{equ:odegroup2kl} and \eqref{equ:odegroup3kl}, we compute
\begin{align*}
X_+X_+'=&\sum_{k+\ell\leq s}\left( 2^{-1}m \sum_{i\in I_1}\left(  |\xi^{(k,\ell)}_{i,1} |^2+|\xi^{(k,\ell)}_{i,2} |^2 \right)+\sum_{i\in\mathbb{N}} \Gamma_i |\xi^{(k,\ell)}_i|^2+m\sum_{1\leq j\leq J}|\bar{z}^{(k,\ell)}_j |^2 \right)\\
&+ \sum_{k+\ell\leq s}\left(  \sum_{i\in I_1}\left(   \xi^{(k,\ell)}_{i,1}\mathcal{E}^{(k,\ell)}_{i,1}  +\xi^{(k,\ell)}_{i,2}\mathcal{E}^{(k,\ell)}_{i,2}  \right)+\sum_{i\in\mathbb{N}}   \xi^{(k,\ell)}_i\mathcal{E}^{(k,\ell)}_i + \sum_{1\leq j\leq J} \bar{z}^{(k,\ell)}_j \overline{\mathcal{W}}^{(k,\ell)}_j  \right).
\end{align*}
Denote the terms on the second line by $X_+(t)Y_+(t)$ and let $b$ be the minimum among $2^{-1}m$ and $|\Gamma_i|,\ i\in\mathbb{Z}\setminus\{0\}$. We have
\begin{align*}
X_+'-bX_+\geq Y_+. 
\end{align*}
Similarly, define $Y_0(t)$ and $Y_-(t)$ by
\begin{align*}
X_0(t)Y_0(t)= \sum_{k+\ell\leq s}\sum_{1\leq j\leq J} z^{(k,\ell)}_j(t) \mathcal{W}^{(k,\ell)}_j(t),
\end{align*}
and
\begin{align*}
X_-(t)Y_-(t)= \sum_{k+\ell\leq s}\sum_{i\in-\mathbb{N} } \xi^{(k,\ell)}_i(t)\mathcal{E}^{(k,\ell)}_i(t).
\end{align*}
It holds that
\begin{align} \label{eq-neutralstable}
X_0'= Y_0,\quad  X_-'+bX_-\leq Y_-.  
\end{align}
We now compare $|X_+(t)|^2+|X_0(t)|^2+|X_-(t)|^2$ and $|Y_+(t)|^2+|Y_0(t)|^2+|Y_-(t)|^2$.
\begin{align*}
&|X_+(t)|^2+|X_0(t)|^2+|X_-(t)|^2\\
&=\sum_{k+\ell\leq s}\left[ |z^{(k,\ell)}(t)|^2+|\bar{z}^{(k,\ell)}(t)|^2+\sum_{i\in I_1}\left(|\xi_{i,1}(t)|^2+|\xi_{i,2}(t)|^2 \right)+\right.\left.\sum_{i\neq 0}|\xi_i(t)|^2 \right]\\
&= \sum_{k+\ell\leq s} \| q^{(k,\ell)}(t) \|^2_G. 
\end{align*}
From the Cauchy-Schwarz inequality and \eqref{def:errorkl}, 
\begin{align*}
&|Y_+(t)|^2+|Y_0(t)|^2+|Y_-(t)|^2\\ 
&\leq  \sum_{k+\ell\leq s}\left( \sum_{i\in I_1}\left( |\mathcal{E}^{(k,\ell)}_{i,1}(t)|^2+|\mathcal{E}^{(k,\ell)}_{i,2}(t)|^2 \right)+\sum_{i\in\mathbb{Z}\setminus\{0\}} |\mathcal{E}^{(k,\ell)}_i(t)|^2+\sum_{1\leq j\leq J}\left( |\mathcal{W}_j(t)|^2+|\overline{\mathcal{W}}_j(t)|^2  \right) \right)\\
&= \sum_{k+\ell\leq s} \| \mathcal{E}^{(k,\ell)}(t) \|^2_G.
\end{align*} 
From \eqref{equ:Ekl} in Corollary~\ref{cor:qE}, $|Y_+(t)|^2+|Y_0(t)|^2+|Y_-(t)|^2=  o(1)\big( |X_+(t)|^2+|X_0(t)|^2+|X_-(t)|^2 \big).$ We can then apply the ODE lemma, Lemma~\ref{lem-MZODE}. In view of \eqref{eq:mz.ode.cor.B1}, the slow decay assumption \eqref{equ:nonexpdecay} rules out the possibility that $X_-(t)$ dominates. Hence  
\begin{equation}\label{equ:ODEmiddle1}
|X_+(t)|^2+|X_0(t)|^2+|X_-(t)|^2=(1+o(1))|X_0(t)|^2. 
\end{equation}
It remains to show that 
\begin{equation}\label{equ:ODEmiddle2}
 |X_0(t)|^2=(1+o(1))|z(t)|^2. 
\end{equation}
For $\ell\geq 1$,
\begin{align*}
z^{(k,\ell)}_j =G(\mathbf{L}q^{k,\ell-1},\Upsilon_j)=G(q^{k,\ell-1},\mathbf{L}^\dagger\Upsilon_j)=0.
\end{align*} 
For $k\geq 1$,
\begin{align*}
z^{(k,\ell)}_j =\frac{d}{dt} z^{(k-1,\ell)}_j=\mathcal{W}^{(k-1,\ell)}_j.
\end{align*}
These imply
\begin{align*}
 |X_0(t)|^2&\leq |z(t)|^2+\sum_{k\leq s-1 }  |\mathcal{W}^{(k,0)}(t) |^2 \\
 &\leq |z(t)|^2+o(1)\left(|X_+(t)|^2+|X_0(t)|^2+|X_-(t)|^2\right)\\
 &\leq |z(t)|^2+o(1)|X_0(t)|^2.
\end{align*}
We used \eqref{equ:ODEmiddle1} in the last inequality. Therefore \eqref{equ:ODEmiddle2} holds and the proof is finished.
\end{proof}

In view of Corollary~\ref{cor:qE} and Lemma~\ref{lem:ODE}, for any $s\in\mathbb{N}_0$,
\begin{align*}
\sum_{k+\ell \le s}\Vert \partial_t^k \slashed{\nabla}^{\ell  }u\Vert_{L^2}(t)=O(1)|z(t)|. 
\end{align*}
Applying the Sobolev embedding on $\Sigma$, we can bound any $C^s$ norm of $u$.
\begin{corollary}[control on higher derivatives]\label{cor:sobolev} For any $s\in \mathbb{N}$, there exists a positive constant $C=C(u,s)$ such that $\Vert u \Vert_{C^{s}}(t) \leq C   |z(t) |. $ 
\end{corollary}
\begin{lemma}\label{lemma-Xi^+_0} There hold the following statements: 
\begin{enumerate} \item $  |z(t)|\geq C^{-1} t^{-1} $ for $t\geq 1$, where $C=C(u)$.
\item For any $\eps>0$ there exists $t_0=t_0(u,	\varepsilon)$, such that on $[t_0,\infty)$, $ e^{-\eps t} |z(t)|$ non-increasing  and $ e^{\eps t} |z(t)|$ non-decreasing.   
\end{enumerate} 

\begin{proof}
The proof directly follows from $ |z'(t)|  \leq C  |z(t)|^2$ and $\lim_{t\to\infty}|z(t)|=0$.  
\end{proof}
\end{lemma}

\subsection{Enhanced decay rate}\label{sec:5.2} Recall that the decomposition of  $u$ in  \eqref{def:ut}  
\begin{equation}\label{def:ut2}
u=u^T+H(u^T)+\tilde{u}^\perp. 
\end{equation} 
Let us introduce an auxiliary quantity
\begin{align}\label{def:Q}
Q(t):= |z(t)|^{p}+|z(t)||\bar{z}(t)|+|z(t)|\left\Vert  u'\right\Vert _{C^2}(t)+ |z(t)|\Vert \tilde u^\perp \Vert _{C^3}(t). 
\end{align}
\begin{lemma}\label{lem:5.5}
There exists a positive constant $C=C(u,\mathcal{M}_\Sigma,N_1 )$ such that 
\begin{equation}\label{equ:MQ}
\left| \mathcal{M}_\Sigma (u)+\nabla f(z(t))-\mathcal{L}_{\Sigma} \tilde{u}^\perp \right| \leq C   Q(t),
\end{equation}
and
\begin{equation}\label{equ:NQ}
|N_1(u)| \leq C  Q(t).
\end{equation}
\begin{proof}
In the proof we use $C$ to represent a positive constant that depends on $u$, $\mathcal{M}_\Sigma$, $N_1$ in \eqref{equ:N}, and its value may vary from one line to another.
Note that the coefficients of $u^T\in \ker \mathcal{L}_{\Sigma}$ are given by $z_j(t)-\bar{z}_j(t)$.
\begin{align*}
u^T(t)= \sum_{j=1}^{J} ( z_j(t) -  \bar{z}_j(t) ) \varphi_{\iota+j}. 
\end{align*}
Fix $\alpha=1/2$. From Lemma~\ref{lem:MSigma},
\begin{align*}
\left| \mathcal{M}_\Sigma (u)+\nabla f(z(t))-\mathcal{L}_{\Sigma} \tilde{u}^\perp \right| \leq C    \Vert u(t) \Vert_{C^{2,\alpha}(\Sigma)}\Vert \tilde{u}^\perp(t)  \Vert_{C^{2,\alpha}(\Sigma)}+  C|\nabla f(z(t)-\bar{z}(t))-\nabla f(z(t)) |.
\end{align*}
Clearly $\Vert u(t) \Vert_{C^{2,\alpha}(\Sigma)}  \leq C  \Vert u \Vert_{C^{3}}(t)$ and $\Vert \tilde{u}^{\perp} (t) \Vert_{C^{2,\alpha}(\Sigma)} \leq C  \Vert \tilde{u}^{\perp} \Vert_{C^{3}}(t)$. From Corollary~\ref{cor:sobolev}, $\Vert u \Vert_{C^{3} }(t) \leq C   |z(t)|$. Therefore,
\begin{align*}
\Vert u(t) \Vert_{C^{2,\alpha}(\Sigma)}\Vert \tilde{u}^\perp(t)  \Vert_{C^{2,\alpha}(\Sigma)}  \leq C  Q(t).
\end{align*}  
Because $f-f(0)$ vanishes at the origin of degree $p$, $|\nabla^2 f(x)| \leq C  |x|^{p-2}$ near the origin. Together with $|\bar{z}(t)| \leq C  |z(t)|$, 
\begin{align*}
|\nabla f(z(t)-\bar{z}(t))-\nabla f(z(t)) |  \leq C   Q(t).
\end{align*}
Hence \eqref{equ:MQ} holds. To show \eqref{equ:NQ}, recall that $N_1(u)=a_1\cdot D u'+a_2\cdot u'+a_3\cdot\mathcal{M}_\Sigma (u)$, where $a_i=a_i(\omega, u,\slashed{\nabla}u,u')$ are smooth with $a_i(\omega,0,0,0)=0$. From Corollary~\ref{cor:sobolev},  $|a_i|  \leq C    |z(t)|$. Hence 
\begin{align*}
|a_1\cdot D u'+a_2\cdot u'|  \leq C  |z(t)|(|Du'|+|u'|)  \leq C  |z(t)|\|u'\|_{C^1}(t) \leq C  Q(t).
\end{align*}
Together with \eqref{equ:MQ}, \eqref{equ:NQ} follows.
\end{proof}
\end{lemma}
 
\begin{lemma}\label{lem:WbarW} There exists a positive constant $C=C(u,m,\mathcal{M}_\Sigma,N_1 )$ such that 
\begin{align}\label{equ:zprimeQ}
|z'(t)+m^{-1}\nabla f(z(t))| \leq C  Q(t),
\end{align}
and
\begin{align}\label{equ:zbarQ}
|\bar{z}'(t)-m\bar{z}(t)+m^{-1}\nabla f(z(t))|  \leq C  Q(t).
\end{align}
\begin{proof}
In view of \eqref{equ:odegroup3}, it suffices to show
\begin{align*}
\left|\mathcal{W}_{j}(t)+ m^{-1} \frac{\partial f}{\partial x^j}(z(t)) \right| \leq C  Q(t)\ \textup{and}\ \left|\overline{\mathcal{W}}_j(t)+m^{-1}\frac{\partial f}{\partial x^j}(z(t)) \right| \leq C Q(t).
\end{align*}
From \eqref{def:error} and $\mathcal{E}(u)= (0,E_1(u))$, we actually have $\mathcal{W}_{j}(t)=\overline{\mathcal{W}}_j(t)$. Recall that
\begin{align*}
\mathcal{W}_j=G(\mathcal{E}(u);\Upsilon_j)=G((0,E_1(u)) ;  ( \varphi_{\iota+j},-2^{-1}m\varphi_{\iota+j}) )=-m^{-1} \int_\Sigma E_1(u)\varphi_{\iota+j}\, d\mu,
\end{align*}
where ${E}_1(u) =N_1(u)-\mathcal{M}_{\Sigma} (u) +\mathcal{L}_{\Sigma}u  $. Because $\varphi_{\iota+j}\in\ker \mathcal{L}_{\Sigma}$,\begin{align*}
\mathcal{W}_j=-m^{-1} \int_\Sigma (N_1(u)-\mathcal{M}_\Sigma(u))\varphi_{\iota+j}\, d\mu,
\end{align*}
Then the assertion follows from \eqref{equ:MQ} and \eqref{equ:NQ}.
\end{proof}
\end{lemma}
\begin{corollary}\label{cor:source1}
Suppose $ \left| Q(t)\right|\leq M |z(t)|^{q}$ for some $M, q>0$. Then there exists a positive constant $C=C(u,m,\mathcal{M}_\Sigma,N_1,M,q)$ such that $$ \left|{z}'(t)\right|  \leq C   |z(t)|^{\min(q,p-1)}\ \textup{and}\ \left| \bar{z}(t)\right|  \leq C  |z(t)|^{\min(q,p-1)}. $$
\begin{proof}
Because $f-f(0)$ vanishes at the origin of degree $p$, $|\nabla f(z(t))| \leq C   |z(t)|^{p-1}$. Hence the bound for $z'(t)$ follows from \eqref{equ:zprimeQ}. The the bound for $\bar{z}(t)$ can be obtained by integrating \eqref{equ:zbarQ}.
\end{proof}
\end{corollary}

We vectorize $\tilde{u}^\perp$ and perform the projection. Set $\tilde{q}:=q(\tilde{u}^\perp)$ 
and
\begin{equation*} 
\begin{split}
&\tilde\xi_{i,1}(t):= G(\tilde q(t),\psi_{i,1}),\ \tilde\xi_{i,2}(t):=G(\tilde q(t),\psi_{i,2})\ \textup{for}\ i\in I_1,\\ 
 &\tilde\xi_{i}(t):= G(\tilde q(t),\Psi_{i})\ \textup{for}\ i\in \mathbb{Z}\setminus\{0\},
\end{split}
\end{equation*}
Note that because $\tilde{u}^\perp$ is orthogonal to $\ker \mathcal{L}_{\Sigma}$, those coefficients completely characterize $\tilde{q}$. The projections of higher order derivatives of $\tilde q$, namely $\tilde q ^{(k,\ell)}={\partial_t^k}\mathbf{L}^\ell \tilde q$, are defined similarly.  In the lemma below, we show that $\{\tilde{\xi}_{i,1},\tilde{\xi}_{i,2}\}_{i\in I_1}$, $\{\tilde{\xi}_i\}_{i\in\mathbb{Z}\setminus\{0\}}$ and the higher order coefficients are bounded by $|z(t)|$.
\begin{lemma}[control on higher derivatives of $\tilde q$] \label{lemma-iterate1}
For any $k,\ell\in\mathbb{N}_0$, there exists a positive constant $C=C(u,m,\ell+k)$ such that

 \begin{align*}
 \sum_{i\in I_1}\left(|\tilde{\xi}^{(k,\ell)}_{i,1}(t)|^2+|\tilde{\xi}^{(k,\ell)}_{i,2}(t)|^2 \right)+ \sum_{i\neq 0}|\tilde{\xi}^{(k,\ell)}_i(t)|^2  \leq C |z(t)|^2.
 \end{align*}

\begin{proof} 
Let $s=\ell+k$.  
\begin{align*}
\sum_{i\in I_1}\left(|\tilde{\xi}^{(k,\ell)}_{i,1}(t)|^2+|\tilde{\xi}^{(k,\ell)}_{i,2}(t)|^2 \right)+ \sum_{i\neq 0}|\tilde{\xi}^{(k,\ell)}_i(t)|^2=\|\tilde{q}^{(k,\ell)}\|^2_G(t)   \leq C  \|\tilde{u}^\perp\|^2_{C^{s+1}}(t).
\end{align*} 
The assertion then follows from Lemma~\ref{lem:utup} (boundedness of decomposition) and Corollary~\ref{cor:sobolev} (control on higher derivatives). 
\end{proof}
\end{lemma}

Let us define
\begin{align}\label{def:tildeE}
&\tilde{E}(u):= \left(\tilde u^\perp\right)''- m\left(\tilde u^\perp\right)' +  \mathcal{L}_{\Sigma} \tilde u^\perp,\text{ and } \tilde{\mathcal{E}}:=(0,\tilde{E}(u)).
\end{align}  
Set $\{\tilde{\mathcal{E}}_{i,1},\tilde{\mathcal{E}}_{i,2}\}_{i\in I_1}$ and $\{\tilde{\mathcal{E}}_i\}_{i\in\mathbb{Z}\setminus\{0\}}$ be the coefficients of $\tilde{\mathcal{E}}$. Namely, 
\begin{equation}\label{def:tildeE_coef} 
\begin{split}
&\tilde{ \mathcal{E}}_{i,1}(t):= G( \tilde{ \mathcal{E}}(t),\psi_{i,1}),\ \tilde{ \mathcal{E}}_{i,2}(t):= G(\tilde{ \mathcal{E}}(t),\psi_{i,2})\ \textup{for}\ i\in I_1,\\
&\tilde{ \mathcal{E}}_{i}(t):= G(\tilde{ \mathcal{E}}(t),\Psi_{i})\ \textup{for}\ i\in \mathbb{Z}\setminus\{0\}.
\end{split}
\end{equation}
Then we have, for $i\in I_1$,
\begin{equation}\label{equ:odegroup1t}
\begin{split}
&\frac{d}{dt} \tilde{\xi}_{i,1}-2^{-1}m \tilde\xi_{i,1}+\beta_i\tilde\xi_{i,2}=\tilde{\mathcal{E}}_{i,1},\\ 
&\frac{d}{dt} \tilde\xi_{i,2}-2^{-1}m \tilde\xi_{i,2}-\beta_i\tilde\xi_{i,1}=\tilde{\mathcal{E}}_{i,2}, 
\end{split}
\end{equation}
for $i\in\mathbb{Z}\setminus\{0\}$,
\begin{equation}\label{equ:odegroup2t}
\frac{d}{dt} \tilde\xi_{i}-\Gamma_i \tilde\xi_{i } =\tilde{\mathcal{E}}_{i }. 
\end{equation}
\begin{lemma}\label{lem:tildeE}
 There exists a positive constant $C=C(u,m,\mathcal{M}_\Sigma,N_1)$ such that
\begin{align*}
\|\tilde{E}(u)\|_{L^2}(t)  \leq C  Q(t) .
\end{align*}
\end{lemma}
\begin{proof}
In the proof we use $C$ to represent a positive constant that depends on $u$, $m$, $\mathcal{M}_\Sigma$, $N_1$ in \eqref{equ:N}, and its value may vary from one line to another.

We compute
\begin{align*}
 \tilde{E}(u) & =\Pi ^\perp \bigg [  E(u)- \left( H(u^T)\right)''  +m\left( H(u^T)\right)'  -   \mathcal{L}_{\Sigma} H(u^T) \bigg] \\
&= \Pi ^\perp \bigg [ -\mathcal{M}_{\Sigma} (u) +\mathcal{L}_{\Sigma} u  +N_1(u) - \left( H(u^T)\right)''  +m\left( H(u^T)\right)'  -   \mathcal{L}_{\Sigma} H(u^T) \bigg]\\
&=:\textsc{I+II} ,  
\end{align*}
 where
 \begin{align*}
 \textsc{I}&= -\Pi^\perp \mathcal{M}_\Sigma(u) + \mathcal{L}_{\Sigma} \tilde{u} ^\perp+\Pi ^\perp N_1(u),\ \textsc{II} =  - \left( H(u^T)\right)''  +m\left( H(u^T)\right)' .
 \end{align*}
From \eqref{equ:MQ} and \eqref{equ:NQ}, we have $\|\textsc{I}\|_{L^2}(t)    \leq C    Q(t).$ From a direct computation,
\begin{align*}
\textsc{II}= -D^2H(z-\bar{z})[ (z-\bar{z})', (z-\bar{z})']-DH(z-\bar{z})[(z-\bar{z})'']+m DH(z-\bar{z})[(z-\bar{z})'].
\end{align*}	
Since $DH(0)=0$, the above is bounded by $$  C |(z-\bar{z})'|^2+ C | z-\bar{z}  ||(z-\bar{z})''|+ C | z-\bar{z}  ||(z-\bar{z})'|   \leq C  \Vert u^T \Vert _{C^2}(t) \Vert \partial _t u ^T \Vert_{C^1}(t).$$
Together with Corollary~\ref{cor:sobolev}, we get \[\|\textsc{II}\|_{L^2}(t)   \leq C  |z(t)| \Vert  u'\Vert_{C^1 }(t)   \leq C  Q(t).\]
\end{proof}

\begin{corollary}\label{cor:source}
Suppose $Q(t) \leq M |z(t)|^q$ for some $M, q>0$. Then there exists a positive constant $C=C(u,m,\mathcal{M}_\Sigma,N_1,M,q)$ such that
\begin{align*}
\sum_{i\in I_1}\left(|\tilde{\xi}_{i,1}(t)|^2+|\tilde{\xi}_{i,2}(t)|^2 \right)+\sum_{i\in\mathbb{Z}\setminus\{0\}}|\tilde{\xi}_{i}(t)|^2  \leq C  |z(t)|^{2q}.
\end{align*}
\end{corollary}
\begin{proof}
In the proof we use $C$ to represent a positive constant that depends on $u$, $m$, $\mathcal{M}_\Sigma$, $N_1$ in \eqref{equ:N}, $M$, $q$, and its value may vary from one line to another.
 We shall prove
\begin{align}\label{equ:sourcemiddle}
 \sum_{i\in -\mathbb{N}}  |\tilde{\xi}_i(t)|^2 \leq C  |z(t)|^{2q},
\end{align} by directly solving evolution equations.  The other terms can be treated similarly.  
For $i\in -\mathbb{N}$, from \eqref{equ:odegroup2t},
\begin{align*}
\tilde\xi_i(t)=e^{\Gamma_i t} \tilde{\xi}(0)+\int_0^t  e^{\Gamma_i(t-\tau)} \tilde{\mathcal{E}}_i (\tau)\, d\tau.
\end{align*}
For any sequence $\{a_i\}_{i\in-\mathbb{N}}$ with $\sum_{i\in -\mathbb{N}} |a_i|^2=1$, we compute
\begin{align*}
\left|\sum_{i\in -\mathbb{N}} \tilde{\xi}_i(t)a_i\right|\leq &\sum_{i\in -\mathbb{N}}  \left| e^{\Gamma_i t}\tilde{\xi}_i(0)a_i\right|+ \sum_{i\in -\mathbb{N}} \int_0^t e^{\Gamma_i(t-\tau)} \left| \tilde{\mathcal{E}}_i (\tau)a_i \right|\, d\tau\\
\leq & e^{-bt} \sum_{i\in -\mathbb{N}}\left| \tilde{\xi}_i(0)a_i\right|+ \sum_{i\in -\mathbb{N}}   \int_0^t e^{-b(t-\tau)} \left| \tilde{\mathcal{E}}_i (\tau)a_i \right|\, d\tau.
\end{align*}
From (1) in Lemma~\ref{lemma-Xi^+_0}, $e^{-bt} \sum_{i\in -\mathbb{N}}\left| \tilde{\xi}_i(0)a_i\right|  \leq C  |z(t)|^q$. From \eqref{def:tildeE}, \eqref{def:tildeE_coef}, Lemma~\ref{lem:tildeE} and the assumption, 
\begin{align*}
\sum_{i\in I_1}\left(|\tilde{\mathcal{E}}_{i,1}(t)|^2+|\tilde{\mathcal{E}}_{i,2}(t)|^2 \right)+\sum_{i\in\mathbb{Z}\setminus\{0\}}|\tilde{\mathcal{E}}_{i}(t)|^2= \Vert \tilde{\mathcal{E}}\Vert^2_G  =2m^{-2}\|\tilde{E}(u)\|^2_{L^2}(t)  \leq C  |z(t)|^q.
\end{align*}
By the Cauchy-Schwarz inequality,
\begin{align*}
\sum_{i\in -\mathbb{N}}   \int_0^t e^{-b(t-\tau)} \left| \tilde{\mathcal{E}}_i (\tau)a_i \right|\, d\tau  \leq C  \int_0^t e^{-b(t-\tau)} |z(\tau)|^q \, d\tau.
\end{align*}
Let $t_0=t_0(u,2^{-1}q^{-1}b)$ be the constant given in (2) of Lemma~\ref{lemma-Xi^+_0}. Then 
\begin{align*}
\int_0^t e^{-b(t-\tau)} |z(\tau)|^q \, d\tau=&\int_0^{t_0} e^{-b(t-\tau)} |z(\tau)|^q \, d\tau +\int_{t_0}^t e^{-2^{-1}b(t-\tau)} \left(e^{-2^{-1}q^{-1}b(t-\tau)} |z(\tau)|\right)^q \, d\tau\\ 
 \leq C  & e^{-bt} \max_{\tau\in [0,\infty)}|z(\tau)|^q+ C |z(t)|^q  \leq C  |z(t)|^q.
\end{align*}
In short, we deduce $\left|\sum_{i\in -\mathbb{N}}  \tilde{\xi}_i(t)a_i\right| \leq C  |z(t)|^q$ for all sequence $\{a_i\}$ with $\sum_{i\in -\mathbb{N}} |a_i|^2=1$. This implies \eqref{equ:sourcemiddle}.
 
\end{proof}

\begin{lemma}[improvement in decay] \label{lem-iterate2} For $1\le r \le p-1 $, suppose the following holds.
\begin{align}\label{equ:iterateamp}
\left| z' (t)\right|^2+|\bar{z}(t)|^2 + \sum_{i\in I_1}\left(|\tilde{\xi}_{i,1}(t)|^2+|\tilde{\xi}_{i,2}(t)|^2 \right)+ \sum_{i\neq 0}|\tilde{\xi}_i(t)|^2  \leq M |z(t)|^{2r}.
\end{align}
Then  for any $\varepsilon\in (0,1)$, there exists a positive constant $C=C(u,m,\mathcal{M}_\Sigma,N_1,M,r,\varepsilon)$ such that  
\begin{align}\label{equ:iterateampnew}
\left| z' (t)\right|^2+|\bar{z}(t)|^2 + \sum_{i\in I_1}\left(|\tilde{\xi}_{i,1}(t)|^2+|\tilde{\xi}_{i,2}(t)|^2 \right)+ \sum_{i\neq 0}|\tilde{\xi}_i(t)|^2 \leq C  |z(t)|^{\min\{2r+2-\varepsilon, 2p-2\}}.
\end{align}
Moreover,
\begin{align}\label{equ:zprimebound}
\left|  z' (t)+m^{-1}\nabla f(z(t)) \right| \leq C  |z(t)|^{ r+1-\varepsilon/2}.
\end{align}
 
\begin{proof}
Fix $\varepsilon\in (0,1)$. In the proof we use $C$ to represent a positive constant that depends on $u$, $m$, $\mathcal{M}_\Sigma$, $N_1$ in \eqref{equ:N}, $M$, $r$, $\varepsilon$, and its value may vary from one line to another. Using the assumption \eqref{equ:iterateamp}, Lemma~\ref{lem:ODE}  (dominance of neutral mode) and Lemma~\ref{lemma-interpolation} (interpolation), we have for all $k,\ell\in\mathbb{N}_0$, $\left|\frac{d}{dt} z^{(k,\ell)} (t)\right|^2+|\bar{z}^{(k,\ell)}(t)|^2 \leq C_{k,\ell}|z(t)|^{2r-\varepsilon }.$ Similarly, applying Lemma \ref{lemma-interpolation} with Lemma~\ref{lemma-iterate1}, 
\begin{align} \label{eq-418tilde}
\sum_{i\in I_1}\left(|\tilde{\xi}^{(k,\ell)}_{i,1}(t)|^2+|\tilde{\xi}^{(k,\ell)}_{i,2}(t)|^2 \right)+ \sum_{i\neq 0}|\tilde{\xi}^{(k,\ell)}_i(t)|^2 \leq C_{k,\ell} |z(t)|^{2r-\varepsilon}.
\end{align}
By the Sobolev embedding, $\|(u^T)'\|_{C^s}(t)+\|\tilde{u}^\perp\|_{C^s}(t)\leq C_s |z(t)|^{r-\varepsilon/2 }$. From \eqref{def:ut2} and \eqref{eq-418tilde}, we also have $\| u' \|_{C^s}(t) \leq C_s |z(t)|^{r-\varepsilon/2 }$. In view of the definition of $Q(t)$ in \eqref{def:Q}, 
$$Q(t) \leq C  |z(t)|^{\min(p, r+1-\varepsilon/2)}= C|z(t)|^{r+1-\varepsilon/2}.$$ 
From Corollary~\ref{cor:source},  
\begin{align}\label{equ:iterateampnew_mid3}
\sum_{i\in I_1}\left(|\tilde{\xi}_{i,1}(t)|^2+|\tilde{\xi}_{i,2}(t)|^2 \right)+ \sum_{i\neq 0}|\tilde{\xi}_i(t)|^2 \leq C  |z(t)|^{2r+2-\varepsilon}.
\end{align}
From Corollary~\ref{cor:source1},  \eqref{equ:zprimebound} and 
\begin{equation}\label{equ:iterateampnew_mid2}
|\bar{z}(t)|^2  \leq C  |z(t)|^{\min\{ 2r+2-\varepsilon ,  2p-2\}}.
\end{equation}
Note that \eqref{equ:zprimebound} implies
\begin{align}\label{equ:iterateampnew_mid1}
|z'(t)|^2  \leq C  |z(t)|^{\min\{2r+2-\varepsilon , 2p-2\}}.
\end{align}
Combining \eqref{equ:iterateampnew_mid3}, \eqref{equ:iterateampnew_mid1} and \eqref{equ:iterateampnew_mid2} yields \eqref{equ:iterateampnew}. 
\end{proof}

\end{lemma}

\begin{proof}[Proof of Proposition~\ref{prop-neutral-dynamics}]
The bound \eqref{equ:zdominates} is a weaker form of Lemma~\ref{lem:ODE} (dominance of neutral mode). From Lemma~\ref{lem:ODE} (dominance of neutral mode) and Lemma~\ref{lemma-iterate1} (control on higher derivatives of $\tilde q$), the assumption~\eqref{equ:iterateamp} in Lemma~\ref{lem-iterate2} (improvement in decay) holds for $r=1$. By iterating Lemma~\ref{lem-iterate2}, \eqref{equ:iterateamp} holds for $r=p-1$. Then the \eqref{equ:zgradient} follows from \eqref{equ:zprimebound}. 
\end{proof}

\section{Slowly decaying solutions to parabolic equation} \label{sec:slowparabolic}
In this section, we consider the slowly decaying solutions to the parabolic equation~\eqref{equ:parabolic}. The main goal is to prove Proposition~\ref{prop-neutral-dynamics-p1}, which is analogous to Proposition~\ref{prop-neutral-dynamics}.

\vspace{0.1cm}

Let $u\in C^\infty(Q_{0,\infty},\widetilde{ \mathbf{V}})$ be a solution to \eqref{equ:parabolic} with $\|u\|_{H^{n+4}}(t)=o(1)$ as $t\to \infty$. From Lemma~\ref{lem:regularity} (parabolic regularity), $\|u\|_{C^s}(t)=o(1)$ for all $s\in\mathbb{N}$. We further assume that $u$ does not decay exponentially. Namely, for any $\varepsilon>0$, 
\begin{equation}\label{equ:nonexpdecay_p1}
\limsup_{t\to \infty} e^{ \varepsilon t}\|u\|_{C^1}(t)=\infty.
\end{equation}
We project $u$ onto the eigensections $\varphi_i$. Set
\begin{align}\label{def:parabolic_xi}
\xi_i(t):=\int_\Sigma  \left\langle u,\varphi_{i}\right\rangle \, d\mu. 
\end{align}
Recall that $\{i\in\mathbb{N}\, :\, \lambda_i=0\}=\{\iota+1,\iota+2,\dots, \iota+J\}$. The neutral mode will play a special role and we denote it by
\begin{align}\label{def:parabolic_x}
x_j(t):= \xi_{\iota+j}(t).
\end{align}

\begin{proposition}\label{prop-neutral-dynamics-p1}
For a slowly decaying solution $u$ to \eqref{equ:parabolic} as described above, there holds
\begin{align}\label{equ:zdominates-p1}
     \sum_{i:\lambda_i\neq 0}|\xi_i(t)|^2  = o(1) |x(t)|^2.
\end{align}
Moreover, for all $\varepsilon>0$, there exists a positive constant $C=C(u,\mathcal{M}_\Sigma, N_2, \varepsilon)$ such that 
\begin{align}\label{equ:zgradient-p1}
 |x'(t)+\nabla f(x(t))| \leq C   |x(t)|^{p-\varepsilon/2}.
\end{align}
\end{proposition} 
The proof of Proposition~\ref{prop-neutral-dynamics-p1} follows a similar structure to the one of Proposition~\ref{prop-neutral-dynamics}. The first part is to show (in Corollary~\ref{cor:parabolic}) that $|x(t)|$ dominates any $C^s$ norm of $u$. 

\bigskip

Let $E_2(u)=N_2(u)+\mathcal{M}_\Sigma (u)-\mathcal{L}_{\Sigma} u$. Then \eqref{equ:parabolic} becomes
\begin{equation}\label{equ:linear_p}
u'-\mathcal{L}_{\Sigma} u=E_2(u).
\end{equation}
From \eqref{equ:quasilinear} and \eqref{equ:N}, the error term $E_2(u)$ has the structure 
\begin{align} \label{equ:E_2str}
E_2(u)=\sum_{j=0}^2 b_{2,j}\cdot \slashed{\nabla}^j u,
\end{align}
where $b_{2,j}=b_{2,j}(\omega, u,\slashed{\nabla}u)$ are smooth with $b_{2,j}(\omega,0,0)=0$. With $\|u\|_{C^s}(t)=o(1)$, the quadratic nature of $E_2$ allows us to bound higher derivatives of $E_2$ in terms of higher derivatives of $u$.
\begin{lemma}\label{lem:error_p}
For any $s\geq 2$, 
\begin{align*}
\sum_{2k+\ell\leq 2s}\| \partial^k_t\slashed{\nabla}^\ell E_2(u) \|_{L^2}(t)=o(1)\sum_{2k+\ell\leq 2s}\| \partial^k_t\slashed{\nabla}^\ell u \|_{L^2}(t).
\end{align*}
\begin{proof}
We present the proof for
\begin{align*}
\sum_{2k+\ell\leq 2s}\left\| \partial^k_t\slashed{\nabla}^\ell \left( b_{2,2}\cdot\slashed{\nabla}^2 u \right)  \right\|_{L^2}(t)=o(1)\sum_{2k+\ell\leq 2s}\| \partial^k_t\slashed{\nabla}^\ell u \|_{L^2}(t).
\end{align*} 
The other terms can be treated similarly. For simplicity, we write $b(\omega,u,\slashed{\nabla}u )$ for $b_{2,2}(\omega, u,\slashed{\nabla}u )$. For $m_0,m_1,m_2 \in\mathbb{N}_0$, we write $b^{(m_0, m_1,m_2 )}$ for the partial derivative of $b$ of order $(m_0,m_1,m_2 )$. Fix $k_0,\ell_0$ with $2k_0+\ell_0\leq 2s$. Then the terms in the expansion of $\partial^{k_0}_t\slashed{\nabla}^{\ell_0} \left( b \cdot\slashed{\nabla}^2 u \right)$ are of the form
\begin{align*}
b^{(m_0,m_1,m_2 )}\cdot\left(\partial^{k_1}_t\slashed{\nabla}^{\ell_1} u \ast \partial^{k_2}_t\slashed{\nabla}^{\ell_2} u \ast\dots \partial^{k_N}_t\slashed{\nabla}^{\ell_N} u \right),
\end{align*} 
where $N=m_1+m_2 +1$, $\sum_{i=1}^N k_i=k_0 $ and $\sum_{i=1}^N \ell_i=\ell_0+m_2+2$. It suffices to show the pointwise bound
$$b^{(m_0,m_1,m_2 )}\cdot\left( \big(\partial^{k_1}_t\slashed{\nabla}^{\ell_1} u  \big)\cdot \big(\partial^{k_2}_t\slashed{\nabla}^{\ell_2} u \big)  \cdot \ldots \cdot \big( \partial^{k_N}_t\slashed{\nabla}^{\ell_N} u\big) \right)=o(1)\sum_{2k+\ell\leq 2s} |\partial^k_t\slashed{\nabla}^\ell u|.$$ We first discuss the case $m_1=m_2=0$. The above becomes $b^{(m_0,0,0)}\cdot \partial^{k_0}_t\slashed{\nabla}^{\ell_0-m_0+2} u $. From $$|b^{(m_0,0,0)}| \leq C ( |u|+|\slashed{\nabla }u|)\ \textup{and}\  |\partial^{k_0}_t\slashed{\nabla}^{\ell_0-m_0+2} u |=o(1),$$ the assertion holds. 

Next, we consider the case $m_1+m_2 \geq 1$. It suffices to show $\min_i (2k_i+\ell_i)\leq 2s$ as the other terms can be bounded by  $o(1)$. Suppose this fails. In other words, $2k_i+\ell_i\geq 2s+1$ for all $1\leq i\leq N$. Then
\begin{align*}
2Ns+N\leq   \sum_{i=1}^N (2k_i+\ell_i)=2k_0+\ell_0+m_2+2\leq 2s+2N. 
\end{align*} 
In view of $N\geq 2$ and $s\geq 2$, this is a contradiction. 
\end{proof}
\end{lemma}
Now we project \eqref{equ:linear_p} onto the eigensections $\varphi_i$. Let
\begin{align*}
\xi^{(k,\ell)}_i(t):=\int_\Sigma  \left\langle \mathcal{L}_{\Sigma}^{\ell}\partial^{k}_t u,\varphi_{i}\right\rangle\, d\mu,\ \mathcal{E}^{(k,\ell)}_i(t):=\int_\Sigma  \left\langle \mathcal{L}_{\Sigma}^{\ell}\partial^{k}_t E_2(u),\varphi_{i}\right\rangle\, d\mu
\end{align*}
Then \eqref{equ:linear_p} becomes
\begin{align}\label{equ:ODE_sys_p}
\frac{d}{dt}\xi^{(k,\ell)}_i(t)-\lambda_i \xi^{(k,\ell)}_i(t)=\mathcal{E}^{(k,\ell)}_i(t).
\end{align}
For the neutral mode, we denote 
\begin{align*}
x^{(k,\ell)}_{j}(t):= \xi^{(k,\ell)}_{\iota+j}(t),\ \mathcal{W}^{(k,\ell)}_{j}(t):=\mathcal{E}^{(k,\ell)}_{\iota+j}(t)
\end{align*}
\begin{lemma}\label{lem:ODE_p}
For any $s\geq 2$,
\begin{align*}
\sum_{k+\ell\leq s}\sum_{i=1}^\infty |\xi^{(k,\ell)}_i(t)|^2=(1+o(1))|x(t)|^2. 
\end{align*}
\begin{proof} Fix $s\geq 2 $. Define three non-negative functions $X_+(t)$, $X_0(t)$ and $X_-(t)$ by
\begin{align*}
X^2_0(t)=&\sum_{k+\ell\leq s}\sum_{1\leq j\leq J}|x^{(k,\ell)}_j(t)|^2,\ X^2_\pm(t)=\sum_{k+\ell\leq s}\sum_{i: \pm\lambda_i>0 }|\xi^{(k,\ell)}_i(t)|^2. 
\end{align*}
We note that these coefficients are grouped together according to the sign of the eigenvalues. From \eqref{equ:ODE_sys_p}, we compute
\begin{align*}
X_+X_+'= &\sum_{k+\ell\leq s}\sum_{i:\lambda_i>0} \left(\lambda_i  |\xi^{(k,\ell)}_i(t)|^2+\xi^{(k,\ell)}_i(t)\mathcal{E}^{(k,\ell)}_i(t)\right)\\
= &\sum_{k+\ell\leq s}\sum_{i:\lambda_i>0} \left(\lambda_i  |\xi^{(k,\ell)}_i(t)|^2 \right)+X_+Y_+,
\end{align*}
here $X_+(t)Y_+(t)$ is defined by the last equality. Let $b=\min\{|\lambda_i|\, :\, \lambda_i\neq 0\}$. We have
\begin{align*}
X_+'-bX_+\geq Y_+. 
\end{align*}
Similarly, define $Y_0(t)$ and $Y_-(t)$ by
\begin{align*}
X_0(t)Y_0(t)= \sum_{k+\ell\leq s}\sum_{1\leq j\leq J} x^{(k,\ell)}_j(t) \mathcal{W}^{(k,\ell)}_j(t),
\end{align*}
and
\begin{align*}
X_-(t)Y_-(t)= \sum_{k+\ell\leq s}\sum_{i:\lambda_i<0 } \xi^{(k,\ell)}_i(t)\mathcal{E}^{(k,\ell)}_i(t).
\end{align*}
It holds that
\begin{align*}
X_0'= Y_0,\quad X_-'+bX_-\leq Y_-.  
\end{align*}
We now compare $|X_+(t)|^2+|X_0(t)|^2+|X_-(t)|^2$ and $|Y_+(t)|^2+|Y_0(t)|^2+|Y_-(t)|^2$.
\begin{align*}
&|X_+(t)|^2+|X_0(t)|^2+|X_-(t)|^2=\sum_{k+\ell\leq s} \| \mathcal{L}_{\Sigma}^{\ell}\partial^{k}_t u  \|_{L^2}^2(t). 
\end{align*}
From the Cauchy-Schwarz inequality, 
\begin{align*}
 |Y_+(t)|^2+|Y_0(t)|^2+|Y_-(t)|^2 &\leq  \sum_{k+\ell\leq s}\sum_{i=1}^\infty |\mathcal{E}^{k,\ell}_i(t)|^2   = \sum_{k+\ell\leq s} \| \mathcal{L}_{\Sigma}^{\ell}\partial^{k}_t E_2(u)  \|_{L^2}^2(t).
\end{align*} 
From Lemma~\ref{lem:error_p},
\begin{align*}
\sum_{k+\ell\leq s} \| \mathcal{L}_{\Sigma}^{\ell}\partial^{k}_t E_2(u)  \|_{L^2}^2(t) =o(1)\sum_{k+\ell\leq s} \| \mathcal{L}_{\Sigma}^{\ell}\partial^{k}_t u  \|_{L^2}^2(t).
\end{align*}
Therefore, $|Y_+(t)|^2+|Y_0(t)|^2+|Y_-(t)|^2=  o(1)\big( |X_+(t)|^2+|X_0(t)|^2+|X_-(t)|^2 \big).$ We can then apply the ODE lemma, Lemma~\ref{lem-MZODE}. The slow decay assumption \eqref{equ:nonexpdecay_p1} rules out the possibility that $X_-(t)$ dominates. Hence \begin{equation}\label{equ:ODE_pmiddle1}
|X_+(t)|^2+|X_0(t)|^2+|X_-(t)|^2=(1+o(1))|X_0(t)|^2. 
\end{equation}
It remains to show that 
\begin{equation}\label{equ:ODE_pmiddle2}
 |X_0(t)|^2=(1+o(1))|x(t)|^2. 
\end{equation}
For $\ell\geq 1$, 
\begin{align*}
x^{(k,\ell)}_j = \int_\Sigma \left\langle \mathcal{L}_{\Sigma}^{\ell}\partial^{k}_t u,\varphi_{\iota+j}\right\rangle\, d\mu=\int_\Sigma \left\langle \partial^{k}_t u, \mathcal{L}_{\Sigma}^{\ell}\varphi_{\iota+j}\right\rangle\, d\mu=0.
\end{align*} 
For $k\geq 1$,
\begin{align*}
x^{(k,\ell)}_j =\frac{d}{dt} z^{(k-1,\ell)}_j=\mathcal{W}^{(k-1,\ell)}_j.
\end{align*}
These imply
\begin{align*}
 |X_0(t)|^2&\leq |x(t)|^2+\sum_{k \leq s-1}  |\mathcal{W}^{(k,0)}(t) |^2 \\
 &\leq |x(t)|^2+o(1)\left(|X_+(t)|^2+|X_0(t)|^2+|X_-(t)|^2\right)\\
 &\leq |x(t)|^2+o(1)|X_0(t)|^2.
\end{align*}
We used \eqref{equ:ODE_pmiddle1} in the last inequality. Therefore \eqref{equ:ODE_pmiddle2} holds and the proof is finished.
\end{proof}
\end{lemma}
\begin{corollary}\label{cor:parabolic}
For any $s\in \mathbb{N}$, there exists a positive constant $C=C(u,s)$ such that $\Vert u \Vert_{C^{s}}(t) \leq C   |x(t) |. $ 
\end{corollary}

The rest of the arguments are simpler than ones in the elliptic case. We omit the details and only provide the main steps. Their elliptic counterparts can be found in Subsection~\ref{sec:5.2}. Let 
$$Q(t)=|x(t)|^p+|x(t)|\Vert u' \Vert_{C^2}(t)+|x(t)|\Vert \tilde{u}^\perp  \Vert_{C^3}(t).$$
\begin{lemma}[c.f. Lemma~\ref{lem:5.5}]\label{lem:3.5}There exists a positive constant $C=C(u,\mathcal{M}_\Sigma,N_{2} )$ such that 
$$|\mathcal{M}_\Sigma(u)+\nabla f(x(t))-\mathcal{L}_\Sigma\tilde{u}^\perp|+|N_2(u)| \leq C  Q(t).$$
\end{lemma} 
\begin{lemma}[c.f. Lemmas~\ref{lem:WbarW} and \ref{lem:tildeE}]\label{lem:3.6}There exists a positive constant $C=C(u,\mathcal{M}_\Sigma,N_2 )$ such that 
$$|x'(t) +\nabla f(x(t))|+\left \Vert\left(\tilde{u}^\perp\right)'-\mathcal{L}_\Sigma\tilde{u}^\perp \right\Vert_{L^2}(t) \leq C Q(t).$$
\end{lemma}

\begin{corollary}[c.f. Corollaries~\ref{cor:source1} and \ref{cor:source}]\label{cor:3.7}
Suppose $Q(t)\leq M |x(t)|^q$ for some $M,q>0$. Then there exists a positive constant $C=C(u ,\mathcal{M}_\Sigma,N_2,M,q)$ such that $|x'(t)| \leq C  |x(t)|^{\min(q,p-1)}$ and $\Vert\tilde{u}^\perp  \Vert_{L^2}(t)  \leq C  |x(t)|^{q}$. 
\end{corollary}

\begin{lemma}[c.f. Lemma~\ref{lem-iterate2}]\label{lem:3.10} For $1\le r \le p-1 $, suppose the following holds.
\begin{align*} 
\left| x' (t)\right| +\Vert \tilde{u}^\perp  \Vert_{L^2}(t) \leq M |x(t)|^{r}.
\end{align*}
Then  for any $\varepsilon\in (0,1)$, there exists a positive constant $C=C(u ,\mathcal{M}_\Sigma,N_2,M,r,\varepsilon)$ such that    
\begin{align*} 
\left| x' (t)\right| +\Vert\tilde{u}^\perp  \Vert_{L^2}(t)  \leq C  |x(t)|^{\min\{ r+1-\varepsilon/2,  p-1\}}.
\end{align*}
Moreover,
\begin{align*} 
\left|  x' (t)+ \nabla f(x(t)) \right|  \leq C  |x(t)|^{ r+1-\varepsilon/2}.
\end{align*}
\end{lemma}

With Lemma~\ref{lem:3.10}, Proposition~\ref{prop-neutral-dynamics-p1} holds through a simple iteration argument. 

\section{Gradient flow}\label{sec:gradient}
In this section, we study the gradient flow on Euclidean space with a perturbative vector field: let $z(t)$ be a curve on $\mathbb{R}^J$ which satisfies \be\label{eq-gf} z'(t)+ \nabla f(z(t))=G(t), \ee 
where $f:\mathbb{R}^J\to \mathbb{R}$ is an analytic potential function and $G(t)$ is a smooth vector field. Let $f$ satisfy the assumptions:
\be f(0)=0, \quad \nabla f(0)=0,\quad  \nabla^2 f(0)=0.\ee 
For some positive integer $p\ge 3$, $f$ has an expansion at zero 
\[f= \sum_{j\ge p} f_j,\]
where $f_j$ is homogeneous polynomial of degree $j$ and $f_p\not\equiv 0$. The restriction of $f_p$ on $\mathbb{S}^{J-1}$ is denoted by $\hat{f}_p$. We write $\slashed{\nabla} $ for the standard connection on $\mathbb{S}^{J-1}$. Consider the critical points and critical values of $\hat{f}_p$ as 
\begin{align*}
\mathbf{C}:=\left\{ \theta\in\mathbb{S}^{J-1}\, :\, \slashed{\nabla}\hat{f}_p(\theta)=0 \right\},\ \mathbf{D}:=\left\{ \hat{f}_p(\theta) \, :\, \theta\in\mathbf{C}  \right\}.
\end{align*} We further assume the perturbative vector field $G$ has a bound  
\begin{equation}\label{eq:Gboundz}
|G(t)| \le C |z(t)|^{p-\varepsilon},
\end{equation}
for some uniform constants $\varepsilon\in (0,2^{-1})$ and $C<\infty$. The main theorem of this section concerns the secant direction when the solution converges to the origin.

\begin{theorem} \label{thm-gradflow}
Suppose $\lim_{t\to\infty} z(t) =0$. Then $|z(t)|^{-p}f(z(t))$ converges to a non-negative critical value $\alpha_0\in\mathbf{D} $. Moreover, one of the following alternatives holds: 
\begin{enumerate}  
\item Suppose $\alpha_0>0$. Then the secant $z(t)/|z(t)|$ converges to a critical point $\theta^*\in\mathbf{C}$ with $\hat{f}_p(\theta^*)=\alpha_0$. Moreover, $\lim_{t\to \infty}t^{ {1}/(p-2)}|z(t)|= ( \alpha_0 p(p-2))^{-1/(p-2)}.$ 
\item Suppose $\alpha_0=0 $. Then 
$\lim_{t\to \infty} \mathrm{dist}(  z(t)/|z(t)|, \mathbf{C}_0)=0$. Here $\mathbf{C}_0$ is some connected component of $\mathbf{C}\cap \left\{\theta : \hat f_p(\theta)= 0\right\}.$ Moreover, 
$\lim_{t\to \infty} t^{ 1/(p-2)}|z(t)|=\infty$.
\end{enumerate}

\end{theorem}

\begin{remark}
The existence of non-negative critical value of $\mathbb{S}^{J-1}$ is a necessary condition for the flow $z(t)$ to converge the the origin. Moreover, if the points in $\mathbf{C}_0$ are isolated, then (2) implies the unique secant limit direction as well. 
\end{remark}
We are ready to prove Theorems~\ref{thm:general_e} and \ref{thm:general_p}.
\begin{proof}[Proof of Theorem~\ref{thm:general_e}]
From Proposition~\ref{prop-neutral-dynamics}, Theorem~\ref{thm-gradflow} applies with the potential function being $m^{-1}f$. Suppose case (1) in Theorem~\ref{thm-gradflow} occurs. This ensures $z(t)/ |z(t)|$ converges to $\theta^*$, a critical point of $\hat{f}_p$ with $m^{-1}\hat{f}_p(\theta^*)=\alpha_0>0$. Moreover, $$\lim_{t\to\infty} t^{1/(p-2)} |z(t)|=\big( \alpha_0  p(p-2) \big)^{-1/(p-2)}.$$
This implies $$\lim_{t\to\infty} t^{1/(p-2)}\Vert u(t) \Vert_{L^2}  =\big( \alpha_0  p(p-2) \big)^{-1/(p-2)}.$$ 
Let $w\in\ker \mathcal{L}_{\Sigma}$ be the section corresponding to $\theta^*$ through \eqref{equ:identification}. Clearly $ u(t)/\Vert u(t) \Vert_{L^2}$ converges to $w$ in $L^2(\Sigma;\mathbf{V})$. From Corollary~\ref{cor:sobolev}, for any $k\in\mathbb{N}$, $ u(t)/\Vert u(t) \Vert_{L^2}$ is uniformly bounded in $C^k(\Sigma;\mathbf{V})$. Therefore $ u(t)/\Vert u(t) \Vert_{L^2} $ converges to $w$ in $C^\infty(\Sigma;\mathbf{V})$. We then obtain case (1) in Theorem~\ref{thm:general_e}. The other possibility, case (2) in Theorem~\ref{thm-gradflow}, leads to case (2) in Theorem~\ref{thm:general_exponential_e}.
\end{proof}
\begin{proof}[Proof of Theorem~\ref{thm:general_p}]
Theorem~\ref{thm:general_p} can be obtained through replacing Proposition~\ref{prop-neutral-dynamics} in the above proof by Proposition~\ref{prop-neutral-dynamics-p1}.
\end{proof}

The rest of the section is devoted to proving Theorem~\ref{thm-gradflow}. Let us write the problem in terms of the polar coordinates $z = r \theta $ where 
$r= |z|$ and $\theta =   z / |z|\in \mathbb{S}^{J-1}$. Because $f_p$ is homogeneous of degree $p$, $f_p(z)=r^p\hat{f}_p(\theta)$. We compute the gradient $\nabla f_p=p r^{p-1}\hat f _p\tfrac{\partial }{\partial r} + r^{p-2} \slashed{\nabla} \hat f_p$.
Let $G^{\perp}(t)$ and $G^T(t)$ be the radial and tangential parts of $G(t)$. Then equation~\eqref{eq-gf} can be decomposed into the radial and tangential parts: 
\bea \label{eq-maingradabsorbed}\begin{cases} \ba 
 &r'(t) = -  r^{p-1}(t)\left(p \hat f_p(\theta(t))+R^{\perp}(t)\right),\\
&\theta'(t) = -r^{p-2}(t)\left( \slashed{\nabla}\hat f_p(\theta(t))+R^T(t) \right).
 \ea \end{cases}\eea 
Here
\begin{align*}
R^{\perp}(t)=&r^{-p+1}(t) G^{\perp}(t)+\sum_{j\geq p+1} j r^{j-p}(t) \hat f_j ,\\ 
R^{T}(t)= &r^{-p+1}(t) G^{T}(t)+ \sum_{j\geq p+1} r^{j-p}(t) \slashed{\nabla}\hat f_j .
\end{align*}
From \eqref{eq:Gboundz}, there exists a constant $A>0$ such that
\begin{align}\label{eq:Rbound}
|R^{\perp}(t)|+|R^{T}(t)|\leq  A r^{1-\varepsilon}(t).
\end{align}
\begin{lemma}\label{lem:limitinf}
$\liminf_{t\to \infty} |\slashed{\nabla} \hat f_p(\theta(t)) |=0. $
\begin{proof} 
It is convenient to work with $\sigma(t)=r^{2-p}(t)$. The equation \eqref{eq-maingradabsorbed} becomes \bea \begin{cases}\label{eq-qtheta} \ba 
 &\sigma'(t) = (p-2) \left(p\hat f_p(\theta(t))-   R^\perp(t)    \right) \\
&\theta'(t) = -\sigma^{-1}(t)  \left(  \slashed{\nabla}\hat f_p(\theta(t))+  R^T(t)   \right).  
 \ea \end{cases}\eea 
The assumption $\lim_{t\to \infty} r(t)=0$ becomes $\lim_{t\to \infty} \sigma(t) = \infty$. 
Suppose $\liminf_{t\to \infty} |\slashed{\nabla} \hat f_p(\theta(t)) |=c_0>0. $ From \eqref{eq:Rbound}, there exists $t_0>0$ such that $|R^T(t)|\leq 4^{-1}c_0$ and $|\slashed{\nabla} \hat f_p(\theta(t)) |\geq 2^{-1}c_0$ for all $t\geq t_0$. This implies for all $t\geq t_0$,
\begin{align*}
\frac{d}{dt} \hat f_p(\theta(t)) =  \slashed{\nabla} \hat f_p(\theta(t)) \cdot \theta'(t) \le  -\frac{c_0^2}{4\sigma(t)} .
\end{align*}
From the first equation in \eqref{eq-qtheta}, we infer that $\sigma$ grows at most linearly. As a result,
\[\hat f_p(\theta(t)) \le \hat f_p(\theta(t_0))+\int _{t_0}^t -\frac{c_0^2}{4\sigma(\tau)}d\tau  \to -\infty \text{ as }t\to \infty, \] which is a contradiction.  
\end{proof}
\end{lemma}
By Lemma~\ref{lem:limitinf}, there exists a sequence $t_i$ goes to infinity such that $|\slashed{\nabla} \hat f_p(\theta(t_i)) |$ goes to zero. Since $\mathbb{S}^{J-1}$ is compact, there exists a subsequence (still denoted by $t_i$) such that $\theta(t_i)$ converges to a critical point $\theta^*\in\mathbf{C}.$ Let $\alpha_0=\hat{f}_p(\theta^*)$. 
\begin{lemma} \label{lem_hatfp} We have $\alpha_0\ge 0$ and $\lim_{t\to\infty} \hat f_p(\theta(t))=\alpha_0$.
\begin{proof}Suppose $\alpha_0<0$. From \eqref{eq:Rbound} and the first equation of \eqref{eq-maingradabsorbed}, $r(t)$ is strictly increasing for $t$ large. This contradicts the assumption $\lim_{t\to\infty} z(t)=0$.

To show the second assertion, let us assume $\hat{f}_p$ is not a constant as otherwise the assertion is trivial. By the {\L}ojasiewicz gradient inequality \cite{Loj}, for each critical point $\theta_0\in\mathbf{C}$, there exists a neighborhood $U$ such that for all $\theta\in \mathbf{C}\cap U$, $ \hat{f}_p(\theta)=\hat{f}_p(\theta_0)$. Therefore, $\mathbf{D}$ is a finite set. Because $\hat{f}_p$ is not a constant, $\mathbf{D}$ has at least two distinct elements. Fix small $\e_0>0$ such that the critical values are separated at least by $6\eps_0$. That is,
$$ \inf _{\alpha_1,\alpha_2\in\mathbf{D},\ \alpha_1\neq \alpha_2 }|\alpha_1-\alpha_2 |\ge 6\e_0 .$$
For every $\e \in (0,\e_0)$, let
\begin{align*}
\delta(\varepsilon):= \inf\left\{  |\slashed{\nabla}\hat{f}_p(\theta)|\, :\, |f(\theta)-\alpha |\geq \varepsilon\ \textup{for all}\ \alpha \in\mathbf{D} \right\}.
\end{align*} 
It is clear that $\delta(\varepsilon)>0$. 
Now assume $\lim_{t\to\infty}\hat{f}_p(\theta(t))\neq \alpha_0$. There exist $\e \in (0,\e_0)$ and $t_i \to \infty$ such that $|\hat f_p(\theta(t_i)) - \alpha_0 | \ge 2\eps .$ Since $\hat f_p(\theta(t))$ is a continuous and differentiable function of $t$, there exist $\tilde t_i \to \infty$ such that 
\[ |\hat f_p (\theta(\tilde t_i))- \alpha_0|=\eps\quad   \text{ and }\quad  \frac{d}{dt} \hat f_p(\theta(\tilde t_i)) \ge 0.\]
From \eqref{eq:Rbound}, for $i$ large enough, $|R^T(\tilde{t}_i)|\leq 2^{-1}\delta(\varepsilon)$. Then we compute from \eqref{eq-qtheta} 
\bea \frac{d}{dt} \hat f_p(\theta(\tilde t_i)) & =  -\sigma^{-1}(\tilde{t}_i) \slashed{\nabla} \hat f_p (\theta(\tilde t_i))\cdot \left(\slashed{\nabla}  \hat{f}_p (\theta(\tilde t_i))+R^T(\tilde{t}_i) \right) \\ 
& \le -2^{-1} \sigma^{-1}(\tilde t_i)\delta^2(\varepsilon)<0. \eea 
This is a contradiction. 
\end{proof}
\end{lemma}

\begin{lemma}\label{lem-distgradflow}
Let $\mathbf{C}_0$ be the component of $\mathbf{C}$ containing $\theta^*$. Then 
\[\lim_{t\to \infty} \mathrm{dist} (\theta(t),\mathbf{C}_0)=0.\]
\begin{proof}
Let $N_\delta$ denote the (open) $\delta $-neighborhood of $\mathbf{C}_0$. If the assertion is false, there is a small $\delta>0$ such that for any $t_0\geq 0$, there exists $t_1\geq t_0$ such that $\theta(t_1)\notin N_{2\delta}$. By taking $\delta>0$ small enough, we may assume $N_{3\delta}\cap\mathbf{C}=\mathbf{C}_0$. Set $\mathbf{A}=\mathrm{cl}(N_{2\delta}\setminus N_\delta)$. Because $\theta^*$ is a limit point of $\theta(t)$, there exist a sequence of strictly increasing times $t_1<t_2<\ldots $ such that $\theta(t)\in \mathbf{A}$ for $t\in [t_{2k-1},t_{2k}]$, $\mathrm{dist}(\theta(t_{2k-1}), \mathbf{C}_0)=2\delta$, and $\mathrm{dist}(\theta(t_{2k}), \mathbf{C}_0)=\delta$ for all positive integers $k$. Since $\mathbf{A}$ is a closed set with $\mathbf{A}\cap\mathbf{C}=\varnothing$, there exists $c>0$ such that $|\slashed{\nabla} \hat f_p(\theta)|\ge c  $ for all $\theta\in \mathbf{A} $. From \eqref{eq:Rbound}, $|R^T(t)|\leq 2^{-1}c $ for $t\ge \bar   t$ some large $\bar t$. This implies for any $t\in [t_{2k-1},t_{2k}]$ with $t_{2k-1}\ge \bar t$, 
\begin{align*}
-\frac{d}{dt}\hat{f}_p(\theta(t))=&\sigma^{-1}(t)\slashed{\nabla} \hat f_p(\theta(t))\cdot\left( \slashed{\nabla} \hat f_p(\theta(t))+R^T(t) \right) \geq  2^{-1}\sigma^{-1}(t)| \slashed{\nabla} \hat f_p(\theta(t))|^2\\
 \geq & 3^{-1}| \slashed{\nabla} \hat f_p(\theta(t))||\theta'(t)|\geq 3^{-1}c  |\theta'(t)|.
\end{align*}  
Therefore,
\begin{align*}
 \hat f_p (\theta( t_{2k}))-\hat f_p (\theta(t_{2k-1})) \leq - 3^{-1}c  \int_{t_{2k-1}}^{t_{2k}}|\theta'(t)|\, dt\leq -3^{-1}c  \delta .
\end{align*}
In particular, $\hat f_p(\theta(t))$ does not converge. This contradicts to Lemma \ref{lem_hatfp}.

\end{proof}

\end{lemma}

\begin{proposition} \label{prop-gradflowlimit}
If $\alpha_0>0$,  then $\lim_{t\to \infty} \theta(t)=\theta^*$. Moreover, the length of curve $\theta(t)$ on $\mathbb{S}^{J-1}$ is finite. 
\begin{proof}
Recall $\alpha_0=\hat f_p(\theta^*).$ Consider the control function 
\[g(t):= r(t) +(\hat f_p(\theta(t)) -\hat f_p(\theta^*)).\]
Our goal is to show the following. There exist $t_0\geq 0$, $c_0>0$, $\rho_0 \in (0,1)$ and a neighborhood $U_0$ of $\theta^*$ such that $g'(t)
\le 0 $, for $t\ge t_0$, and moreover, 
\begin{equation}\label{eq-controlg}
 g'(t)\le  - c_0 |g(t)|^{\rho_0} |\theta'(t)|  \quad \text{ if } \ t\ge t_0 \text{ and } \theta(t)\in U_0. \end{equation} 

Let us first show this implies the proposition. First, since $g(t)$  decreases  monotonically to zero, $g(t)\geq 0$ for $t\geq t_0$. Suppose $\theta(t)\in U_0$ for $t\in [t_1,t_2]$ with $t_1\ge t_0$. By integrating \eqref{eq-controlg}, 
\[  \int_{t_1}^{t_2} |\theta'(t)|dt   \le  \frac{|g(t_1)|^{1-\rho_0}-|g(t_2)|^{1-\rho_0}}{c_0(1-\rho_0 )}  \le  \frac{|g(t_1)|^{1-\rho_0}}{c_0(1-\rho_0 )}  .\] 
 This shows that if $t_1\ge t_0$, $g(t_1)$ is sufficiently small and $\theta(t_1)$ is sufficiently close to $\theta^*$ (we may find such $t_1$ by Lemma \ref{lem_hatfp} and $\lim_{i\to \infty} \theta(t_i) = \theta^*$), then $\theta(t)$ remains inside of $U_0$ for all $t\ge t_1$ and the length of the curve $\theta(t)$ for $t\in [ t_1,\infty)$ is finite. In particular, $\theta(t)$ converges as $t\to \infty$ and this proves the proposition.

It remains to prove $g'(t)\le0$ and \eqref{eq-controlg}. Let $A$ be the constant in \eqref{eq:Rbound}. Fix $t_0\geq 0$ large so that for  $t\geq t_0$, 
\begin{align}
&p\hat{f}_p(\theta(t))+R^{\perp}(t)\geq  2^{-1}p\hat{f}_p(\theta^*), \label{eq:t0_1} \\
&r^{1-2\varepsilon}(t)\leq   (160)^{-1} A^{-2} p\hat{f}_p(\theta^*). \label{eq:t0_2} 
\end{align}
This is possible because of Lemma~\ref{lem_hatfp} and $\hat{f}_p(\theta^*)>0$. Observe from \eqref{eq-maingradabsorbed}, 
\be\label{eq-g'evol} g'(t)=-r^{p-1}(t) \left(p\hat{f}_p(\theta^*) +R^{\perp} \right) -r^{p-2}(t) \slashed{\nabla}\hat{f}_p(\theta(t))\cdot \left(\slashed{\nabla}\hat{f}_p(\theta(t))+ R^T(t)\right) .\ee
For $t\geq t_0$, 
\begin{align*}
g'(t)
\leq & -2^{-1}pr^{p-1}(t) \hat{f}_p(\theta^*) -r^{p-2}(t) \slashed{\nabla}\hat{f}_p(\theta(t))\cdot \left(\slashed{\nabla}\hat{f}_p(\theta(t))+ R^T(t)\right) \qquad \text{by \eqref{eq:t0_1}}\\
\leq & -2^{-1}pr^{p-1}(t) \hat{f}_p(\theta^*)+4^{-1}r^{p-2}(t)|R^T(t)|^2\qquad \text{by  Cauchy-Schwarz}\\
\leq & -2^{-1}pr^{p-1}(t) \hat{f}_p(\theta^*)+4^{-1}A^2r^{p-2\varepsilon}(t)\leq 0 \qquad  \text{by \eqref{eq:Rbound}}.
\end{align*}
This shows $g'(t)\le0$. 

Next, applying the {\L}ojasiewicz  gradient inequality \cite{Loj} to $\hat f_p$ by viewing it as an analytic function $|x|^{-p} f_p(x)$ on $\mathbb{R}^J\setminus \{0\}$, we obtain a neighborhood $U_0$ of $\theta^*$ in $\mathbb{S}^{J-1}$, $\rho_1\in(0,1)$ and $c_1>0$ such that if $\theta\in U_0$ then 
\begin{equation}\label{eq:Loj}
|\slashed{\nabla} \hat f_p (\theta)|\ge c_1 |\hat f_p(\theta)-\hat f_p(\theta^*)|^{\rho_1}.
\end{equation}
Now we suppose $\theta(t)\in U_0$ for some $t\ge t_0$ and show that \eqref{eq-controlg} holds. We divide the discussion into two cases depending on the term we utilize from the right hand side of \eqref{eq-g'evol}. 

\noindent \textbf{Case 1}. Suppose $|\slashed{\nabla } \hat f_p(\theta(t)) +R^T(t)|\geq 4Ar^{1- \e }(t)$. This implies
\begin{equation}\label{equ:f3A}
|\slashed{\nabla } \hat f_p(\theta(t))| \geq 3Ar^{1- \e }(t) 
\end{equation}
and
\begin{align}\label{equ:f-RT}
|\slashed{\nabla } \hat f_p(\theta(t))| -|R^T(t)|\geq 2^{-1}|\slashed{\nabla } \hat f_p(\theta(t)) +R^T(t)|=2^{-1}r^{2-p}(t)|\theta'(t)|. 
\end{align}
From \eqref{eq:Loj} and \eqref{equ:f3A},
\begin{align}\label{equ:f3AA}
|\slashed{\nabla} \hat f_p (\theta(t))|\ge 2^{-1}\left( c_1 |\hat f_p(\theta(t))-\hat f_p(\theta^*)|^{\rho_1}+  3Ar^{1- \e }(t)\right).
\end{align}
Combining \eqref{equ:f-RT} and \eqref{equ:f3AA}, we derive
\begin{align*}
g'(t)\leq & -r^{p-2}(t) |\slashed{\nabla} \hat f_p (\theta(t))|\left(|\slashed{\nabla } \hat f_p(\theta(t))| -|R^T(t)|\right) \\
 \leq &-4^{-1}\left( c_1 |\hat f_p(\theta(t))-\hat f_p(\theta^*)|^{\rho_1}+  3Ar^{1- \e }(t)\right)|\theta'(t)|.
\end{align*}
As a result, for $\rho_2= \max (\rho_1, 1-  \eps )\in(0,1)$ and some $c_2>0$, 
\[g'(t) \le-c_2 |g(t)|^{\rho_2}|\theta'(t)| .\]

\noindent \textbf{Case 2}. Suppose $|\slashed{\nabla } \hat f_p(\theta(t)) +R^T(\theta(t))|< 4Ar^{1-\varepsilon}(t)$. Then $|\theta'(t)| \leq 4 Ar^{p-1-\varepsilon}(t) .$ This implies
\begin{align*}
-2^{-1}  pr^{p-1}(t) \hat{f}_p(\theta^*)\leq -8^{-1}A^{-1}  p \hat{f}_p(\theta^*) r^{\varepsilon}(t)|\theta'(t)|=-4c_3r^{ \varepsilon}(t)|\theta'(t)|.
\end{align*}
Here $c_3=(32)^{-1}A^{-1}p\hat{f}_p(\theta^*) $ is defined by the last equality. From \eqref{eq:t0_2},
\begin{equation}\label{eq:case2}
 |\slashed{\nabla}\hat f_p(\theta(t)) | \le 5Ar^{1- {\e} }(t) \le c_3 r^{{\e} }(t).
\end{equation}
Hence
\begin{align*}
g'(t)\leq -2^{-1}  pr^{p-1}(t) \hat{f}_p(\theta^*) +\slashed{\nabla}\hat{f}_p(\theta(t))\cdot\theta'(t) \leq   -3c_3r^{\varepsilon}(t)|\theta'(t)|.
\end{align*}
Moreover, from \eqref{eq:case2} and \eqref{eq:Loj}, 
\begin{align*}
3c_3r^{ \varepsilon}(t)\geq 2c_3r^{ \varepsilon}(t)+ |\slashed{\nabla}\hat f_p(\theta(t)) |\geq  2c_3r^{ \varepsilon}(t)+c_1 |\hat f_p(\theta(t))-\hat f_p(\theta^*)|^{\rho_1}\geq c_4 |g(t)|^{\rho_3}
\end{align*}
for some $c_4>0$ and $\rho_3= \max (  \varepsilon  , \rho_1)\in (0,1)$. The inequality \eqref{eq-controlg} is obtained and this finishes the proof. 
\end{proof}

\end{proposition}

\begin{proof}[Proof of Theorem \ref{thm-gradflow}]   By Lemma \ref{lem_hatfp}, there exists a critical point $\theta^*\in\mathbf{C}$ such that
\begin{align*}
\lim_{t\to\infty} |z|^{-p} f(z)=\hat f_p(\theta^*)\ge 0.
\end{align*}
As previously, we denote $\alpha_0=\hat f_p(\theta^*)$. If $\alpha_0>0$, then the convergence of $z(t)/|z(t)|$ is an immediate consequence of Proposition \ref{prop-gradflowlimit}. Moreover, by solving the first line in \eqref{eq-maingradabsorbed} with known asymptotics, we obtain
$$\lim_{t\to\infty} t^{1/(p-2)}|z(t)|=(\alpha_0 p(p-2))^{1/(p-2)}.$$ 
Suppose $\alpha_0=0$. Let $\mathbf{C}_0$ be the connected component of $\mathbf{C}$ which contains $\theta^*$. Since there are finitely many critical values (see the proof of Lemma~\ref{lem_hatfp}) and $\hat{f}_p(\mathbf{C}_0)$ is connected, $\hat{f}_p(\mathbf{C}_0)=\{0\}$. In particular, $\mathbf{C}_0$ is a connected component of $\mathbf{C}\cap\{ \theta\, :\, \hat{f}_p(\theta)=0 \}$. The convergence of $\mathrm{dist}(z(t)/|z(t)|,\mathbf{C}_0)$ follows by Lemma~\ref{lem-distgradflow} and the limit of $t^{1/(p-2)}|z(t)|$ follows by a similar argument.

\end{proof}

\appendix \section{Tools}\label{sec:A}
The following ODE lemma was proved in \cite{MZ}.
\begin{lemma}[Merle-Zaag ODE lemma]\label{lem-MZODE}
Let $X_+$, $X_0$, $X_- : [0,\infty) \to [0, \infty)$ be absolutely continuous functions such that $X_+(t) + X_0(t) + X_-(t) > 0$ for all $t\geq  0$ and $		\liminf_{t \to \infty} X_+(t) = 0.$ Suppose there exist $b>0$ and functions $Y_+$, $Y_0$, $Y_-$ with 
\begin{align*}
|Y_+(t)|^2 + |Y_0 (t)|^2 + |Y_-(t)|^2 = o(1)(|X_+(t)|^2 + |X_0 (t)|^2 + |X_-(t)|^2)
\end{align*} such that
\begin{equation}  \label{eq:mz.ode.cor}
		\begin{array}{ccc}
			X_+'(t) - bX_+(t) & \geq & Y_+(t), \\
			|X_0'(t)| & \leq & Y_0(t), \\
			X_-'(t) + bX_-(t) & \leq &  Y_-(t) .
		\end{array}
	\end{equation}
Then one of the following holds: either
\begin{align}\label{eq:mz.ode.cor.A}
X_+(t)+X_-(t)=o(1)X_0(t),
\end{align}	
or
\begin{align}\label{eq:mz.ode.cor.B}
X_+(t)+X_0(t)=o(1)X_-(t).
\end{align}
Moreover, suppose \eqref{eq:mz.ode.cor.B} holds true. Then for all $\varepsilon>0$,
\begin{align}\label{eq:mz.ode.cor.B1}
\limsup_{t\to\infty} e^{(b-\varepsilon)t}(X_+(t) + X_0(t) + X_-(t))=0.
\end{align}
\end{lemma} 

\begin{lemma}[interpolation]\label{lemma-interpolation} Suppose a function of single variable $g(t)$ has bounds\[\sup_{t\in[-1,1]} \left |\frac{d^j }{dt^j}g(t) \right | \le C_k \] for $j=0,1,\ldots,k$ and \[\sup_{t\in [-1,1]} |g(t)| \le C_0.\] Then, for every $0\leq \ell < k$, there holds a bound 
\[\left |\frac{d^\ell}{dt ^\ell } g (0)\right | \le  \beta \cdot C_0^{\alpha} C_k^{1-\alpha} \]
for some $\alpha=\alpha(\ell,k)>0$ and $\beta= \beta(\ell,k)>0$. Moreover, $\alpha(\ell,k)$ converges to $1$ as $k \to \infty $ while $\ell$ is fixed.

\begin{proof} We prove the inequality with $\beta\equiv 1$ for functions which are compactly supported in the interior of $(-1,1)$. This is sufficient as one may multiply a general function by a cut-off and then apply the result. 

By an argument which uses an induction and the integration by parts, we obtain for every $\ell< k$
\[ \Vert \partial^\ell _t g \Vert _{L^2} \le \Vert  g \Vert _{L^2}^{1-\frac \ell k} \Vert \partial ^k _t g\Vert_{L^2} ^{\frac{\ell}{k}}   .\]
\[\sup_{t\in[ -1,1]} |\partial ^{\ell}_t g| \le \Vert \partial^{\ell+1} _t g\Vert_{L^1} \le\Vert \partial^{\ell+1} _t g\Vert_{L^2}\le \Vert  g \Vert _{L^2}^{1-\frac {\ell+1}k} \Vert \partial ^k _t g\Vert_{L^2} ^{\frac{\ell+1}{k}}  \le C_0^\alpha C_{k}^{1-\alpha} ,\]
where $\alpha = 1- \frac{\ell+1}k$. 
\end{proof}

\end{lemma}

In the next lemma, we record the elliptic regularity in \cite[(1.13)]{S0}.

\begin{lemma}[elliptic regularity]\label{lem:regularity_e}
There exists $\rho_0>0$ small depending on $\mathcal{M}_\Sigma$ such that the following holds. Let $u$ be a solution to \eqref{equ:main} for $t\in [T-2,T+2]$. Assume $\Vert u \Vert_{C^1}(t)\leq \rho_0$ for $t\in [T-2,T+2]$. Then for any $s\in\mathbb{N}$, there exists a positive constant $C=C(m,\mathcal{M}_\Sigma,N_1,s)$ such that
\begin{align*}
\sup_{t\in [T-1,T+1]}\Vert u \Vert^2_{C^s}(t) \leq C  \int_{T-2}^{T+2} \Vert u \Vert_{L^2}^2(t)\, dt.
\end{align*} 
\end{lemma}

The parabolic regularity below is a minor extension of the one presented in \cite{S0}. We include the proof for readers' convenience. 

\begin{lemma}[parabolic regularity]\label{lem:regularity}
Let $u\in C^\infty(Q_{0,\infty},\widetilde{\mathbf{V}})$ be a solution to \eqref{equ:parabolic}. Assume $\lim_{t\to \infty} \Vert u \Vert_{H^{n+4}}(t)=0$. Then for all $s\in \mathbb{N}$, $\lim_{t\to \infty} \Vert u \Vert_{C^s}(t)=0$.
\begin{proof}
We claim that that for $\ell\geq n+4$,
\begin{equation}\label{equ:regularity}
 \lim_{t\to \infty} \Vert u \Vert_{H^{\ell}}(t)=0\ \textup{ implies}\  \lim_{t\to \infty} \Vert u \Vert_{H^{\ell+1}}(t)=0. 
\end{equation}
Suppose for a moment \eqref{equ:regularity} holds true. Applying the Sobolev embedding on $\Sigma$,  $u(t)$, as functions on $\Sigma$, converges smoothly to zero. In view of \eqref{equ:linear_p}, $u'(t)$ and higher order time derivatives also converge smoothly to zero.

We then turn to proving \eqref{equ:regularity}.  From \eqref{equ:linear_p} and \eqref{equ:H1}, we have the following energy estimate. For any $\ell\in\mathbb{N}$, there exists $\delta_\ell>0$ such that for $t_2>t_1>0$ with $t_2-t_1\leq \delta_\ell$, there holds that
\begin{align}\label{equ:regularity3}
\sup_{t\in [t_1,t_2]} \Vert u \Vert_{H^\ell}(t)+\int_{t_1}^{t_2} \Vert u \Vert_{H^{\ell+1}}(t)\, dt\leq C_\ell  \Vert u \Vert_{H^\ell}(t_1)+C_\ell\int_{t_1}^{t_2} \Vert E_2(u) \Vert_{H^{\ell-1}}(t)\, dt.
\end{align}
We omit the derivation and instead refer readers to equation (4.3) in \cite{S0}. Assume $\lim_{t\to \infty} \Vert u \Vert_{H^{\ell}}(t)=0 $. We claim that 
\begin{equation}\label{equ:regularity2}
\Vert E_2(u) \Vert_{H^{\ell-1}}(t)=o(1)\Vert u \Vert_{H^{\ell+1}}(t)\ \textup{and}\ \Vert E_2(u) \Vert_{H^{\ell}}(t)=o(1)\Vert u \Vert_{H^{\ell+2}}(t). 
\end{equation}
Suppose for a moment \eqref{equ:regularity2} holds true. Fix $\delta=\min(\delta_{\ell},2^{-1}\delta_{\ell+1})$. Applying \eqref{equ:regularity3} and absorbing the $E_2(u)$ terms using \eqref{equ:regularity2}, for $t_1$ large enough, 
\begin{align*}
\int_{t_1}^{t_1+\delta } \Vert u \Vert_{H^{\ell+1}}(t)\, dt \leq C_\ell  \Vert u \Vert_{H^\ell}(t_1),\ \sup_{t\in [t_1,t_1+2\delta]} \Vert u \Vert_{H^{\ell+1}}(t)\leq C_{\ell+1} \Vert u \Vert_{H^{\ell+1}}(t_1). 
\end{align*}
Then \eqref{equ:regularity} follows.

It remains to prove \eqref{equ:regularity2}. Recall that \eqref{equ:E_2str}  $E_2(u)=\sum_{j=0}^2 b_{2,j}\cdot \slashed{\nabla}^j u$, where $b_{2,j}=b_{2,j}(\omega, u,\slashed{\nabla}u)$ are smooth with $b_{2,j}(\omega,0,0)=0$. We present the proof for
\begin{align}\label{equ:regularity4}
 \left\Vert   b_{2,2}\cdot\slashed{\nabla}^2 u    \right\Vert_{H^{\ell}}(t)=o(1)\Vert u \Vert_{H^{\ell+2}}(t) .
\end{align} 
The other terms can be treated similarly. For simplicity, we write $b(\omega, u,\slashed{\nabla}u)$ for $b_{2,2}(\omega,u,\slashed{\nabla}u)$. For $m_0, m_1,m_2\in\mathbb{N}_0$, we write $b^{(m_0,m_1,m_2)}$ for the partial derivative of $b$ of order $(m_0, m_1,m_2)$. Fix $\ell_0 \leq \ell$. Then the terms in the expansion of $ \slashed{\nabla}^{\ell_0} \left(   b\cdot\slashed{\nabla}^2 u  \right)$ are of the form
\begin{align*}
b^{(m_0, m_1,m_2)}\cdot\left( \slashed{\nabla}^{\ell_1} u \ast  \slashed{\nabla}^{\ell_2} u \ast\dots  \slashed{\nabla}^{\ell_N} u \right),
\end{align*} 
where $N=m_1+m_2+1$ and $m_0+ \sum_{i=1}^N \ell_i=\ell_0+m_2+2$. Furthermore, $\ell_i\geq 1$ for all $1\leq i\leq N$, and there exist at least $m_2+1$ values of $i$ for which $\ell_i\geq 2$. We first discuss the case $m_1=m_2=0$. Then the above becomes $b^{(m_0,0,0)}\cdot  \slashed{\nabla}^{\ell_0-m_0+2} u $. From the Sobolev embedding,
\begin{align*}
|b^{(m_0,0,0)}|\leq C( |u|+|\slashed{\nabla}u|) \leq C\Vert u \Vert_{H^{n+4}}(t)=o(1) .
\end{align*}
Hence 
$$ \left\Vert b^{(m_0,0,0)}\cdot  \slashed{\nabla}^{\ell_0-m_0+2} u \right\Vert_{L^2}(t) =o(1) \left\Vert u \right\Vert_{H^{\ell+2}}(t).$$

Next, we consider the case $m_1+m_2\geq 1$. It is then straightforward to check that the maximum of $\{\ell_i\}_{1\leq i\leq N}$ is at most $\ell+2$ and other elements are at most $\lfloor \frac{\ell+3}{2}\rfloor =\ell_*$. From $\ell\geq n+4$ and the Sobolev embedding,    
$\left\Vert \slashed{\nabla}^{\ell_*} u(t)  \right\Vert_{L^\infty(\Sigma)}\leq C \left\Vert u \right\Vert_{H^\ell}(t)=o(1).$ This implies
\begin{align*}
\left\Vert b^{(m_0,m_1,m_2)}\cdot\left( \slashed{\nabla}^{\ell_1} u \ast  \slashed{\nabla}^{\ell_2} u \ast\dots  \slashed{\nabla}^{\ell_N} u \right)\right\Vert_{L^2}(t) =o(1) \left\Vert u \right\Vert_{H^{\ell+2}}(t),
\end{align*} 
and finishes the proof of \eqref{equ:regularity4}.
\end{proof} 
\end{lemma}

\section{Mean curvature flow}\label{sec:B}

In this section, we show that (rescaled) mean curvature flows close to a stationary solution can be described by \eqref{equ:parabolic}.

\vspace*{0.3cm}

We begin with the mean curvature flow (MCF). Let $\Sigma$ be an embedded closed $n$-dimensional submanifold in an $n+k$-dimensional analytic ambient space $(M,\bar{g})$. Let $d\mu$ be the volume form induced from the metric on $\Sigma$. Let $\mathbf{V}$ be the normal bundle of $\Sigma$ equipped with the inner product and the connection induced from $M$. For $\varepsilon>0$, let $\mathcal{U}_\varepsilon=\{u\in\mathbf{V}: |u|<\varepsilon\}$ and let $N_\varepsilon (\Sigma)$ be the $\varepsilon$-tubular neighborhood of $\Sigma$ in $M$. For $\varepsilon$ small enough, $\mathcal{U}_\varepsilon$ and $N_\varepsilon(\Sigma)$ can be identified through the exponential map:
\begin{align}\label{equ:exp}
\exp :\mathcal{U}_\varepsilon\to N_\varepsilon(\Sigma).
\end{align}

Consider a local coordinate chart of $\mathbf{V}$ as $\{x^i, y^A\}$, $ 1\leq i\leq n$, $1\leq A\leq k$, where $x^i$ are coordinates of $\Sigma$ and $y^A$ are fiber coordinates. We write $(g_0)_{ij}$ and $(h_0)_{AB}$ for the metric on $\Sigma$ and the inner product on $\mathbf{V}$, respectively.  Under the identification \eqref{equ:exp}, the ambient metric $\bar{g}$ is of the form
\begin{align*}
\bar{g}=g_{ij}dx^idx^j+h_{AB}dy^Ady^B+ c_{iA} (dx^idy^A+dy^Adx^i).
\end{align*}
When $y^A= 0$, $g_{ij}=(g_0)_{ij}$, $h_{AB}=(h_0)_{AB}$ and $c_{iA}=0$. In particular, the induced volume form of $\Sigma$ is given by
\begin{align*}
d\mu=\sqrt{\det g_0} dx^1\wedge\dots \wedge dx^n.
\end{align*}

Let $u$ be a section of $\mathbf{V}$ which is contained in $\mathcal{U}_\varepsilon$. From \eqref{equ:exp}, $u$ can be identified as a submanifold in $N_\varepsilon (\Sigma)$ and we denote this submanifold by $\textup{graph}(u)$. The induced metric on $\textup{graph}(u)$ is
\begin{align*}
 (g_u)_{ij}=g_{ij}+\frac{\partial u^A}{\partial x^i}\frac{\partial u^B}{\partial x^j}h_{AB}+\frac{\partial u^A}{\partial x^i} c_{jA}+\frac{\partial u^B}{\partial x^j} c_{iB}.
\end{align*}
Let $ (g_u)^{ij}$ be the inverse metric of $(g_u)_{ij}$. It is straightforward to check that
\begin{align*}
e_A:= \frac{\partial}{\partial y^A}-(g_u)^{ij}\left( c_{iA}+\frac{\partial u^B}{\partial x^i}h_{BA} \right)\left( \frac{\partial}{\partial x^j}+\frac{\partial u^C}{\partial x^j}\frac{\partial}{\partial y^C} \right),\ 1\leq A\leq k
\end{align*}
form a basis for the normal space of graph$(u)$. Let $ (h_u)_{AB}=\bar{g}(e_A,e_B)$ and $(h_u)^{AB}$ be its inverse metric.  
 
Now we compare the mean curvature vector and the Euler-Lagrange operator of the area functional. The (local) area functional $\mathcal{F}_\Sigma(u)$ is given by
\begin{align*}
\mathcal{F}_\Sigma (u)=\int \sqrt{\det g_u }\,dx^1\wedge\dots\wedge dx^n.
\end{align*} 
Denote by $\vec{H}$ the mean curvature vector of $\textup{graph}(u)$. Let $\xi$ be a smooth section of $\mathbf{V}$. From the first variational formula, \begin{align*}
\frac{d}{ds}\mathcal{F}_\Sigma (u+s\xi)\big|_{s=0}=-\int \bar{g}\left(\vec{H}, \xi^A\frac{\partial}{\partial y^A}  \right) \sqrt{\det g_u }\,dx^1\wedge\dots\wedge dx^n.
\end{align*}
Comparing the above with 
\begin{align*}
\frac{d}{ds}\mathcal{F}_\Sigma (u+s\xi)\big|_{s=0}=-\int h_0\left( \mathcal{M}_\Sigma (u),\xi \right) \sqrt{\det g_0}\,dx^1\wedge\dots\wedge dx^n,	
\end{align*}
we have 
\begin{equation}\label{equ:MH}
\mathcal{M}_\Sigma (u)^A=\bar{g}\left(\vec{H},\frac{\partial} {\partial y^B}\right)(h_0)^{BA} \sqrt{\frac{\det g_u }{\det g_0}}. 
\end{equation} 
In particular, $\mathcal{M}_\Sigma(0)=0$ if $\Sigma$ is a minimal submanifold.  

Next, we compute the non-parametric form of the MCF. In this case, the speed $\eta^A=\frac{\partial u^A}{\partial t}$ is determined by $\eta^A \bar{g}(\frac{\partial}{\partial y^A},e_B)=\bar{g}(\vec{H},e_B)$ for all $1\leq B\leq k$. From $\bar{g}(\frac{\partial}{\partial y^A},e_B)=(h_u)_{AB}$ and $\bar{g}(\vec{H},e_B)=\bar{g}(\vec{H},\frac{\partial}{\partial y^B} )$, we have $\eta^A=  \bar{g}\left( \vec{H},\frac{\partial}{\partial y^B} \right) (h_u)^{BA}$. Hence the non-parametric form of the MCF is given by
\begin{align}\label{equ:MCF}
\frac{\partial u^A}{\partial t}=  \bar{g}\left( \vec{H},\frac{\partial}{\partial y^B} \right) (h_u)^{BA}.
\end{align}
Using \eqref{equ:MH}, \eqref{equ:MCF} is equivalent to
\begin{align}\label{equ:MCF2}
\frac{\partial u^A}{\partial t}-\mathcal{M}_\Sigma(u)^A=\left( \sqrt{\frac{\det  g_0}{\det   g_u}}(h_u)^{AB}(h_0)_{BC}-\delta_C^A  \right)\cdot \mathcal{M}_\Sigma(u)^C.
\end{align}
In view of \eqref{equ:N}, \eqref{equ:MCF2} is of the form \eqref{equ:parabolic}.

\vspace*{0.3cm}
Next, we consider the rescaled MCF close to a shrinker. Let $\Sigma$ be an $n$-dimensional embedded closed submanifold in $\mathbb{R}^{n+k}$. We continue to use notation introduced above. For a section $u$ of the normal bundle $\mathbf{V}$, we write $X(u)$ for the position vector. In particular, $X(0)$ stands for the position vector of $\Sigma$. Let $d\mu$ be the volume form with a Gaussian weight. Namely,
\begin{align*}
d\mu= e^{-|X(0)|^2/4}\sqrt{\det g_0} dx^1\wedge\dots\wedge dx^n.
\end{align*}
For a general section $u$, the (local) Gaussian area functional is  
\begin{align*}
\mathcal{F}_\Sigma(u)=\int  e^{-|X(u)|^2/4}\sqrt{\det g_u} dx^1\wedge\dots\wedge dx^n.
\end{align*}
Let $\xi$ be a smooth section of $\mathbf{V}$. From the first variational formula, 
\begin{align*}
\frac{d}{ds}\mathcal{F}_\Sigma (u+s\xi)\big|_{s=0}=-\int  \left\langle \vec{H}+\frac{X^\perp(u)}{2}, \xi^A\frac{\partial}{\partial y^A}  \right\rangle e^{-|X(u)|^2/4}  \sqrt{\det g_u }\,dx^1\wedge\dots\wedge dx^n,
\end{align*}
where $X^\perp(u)$ stands for the normal part of $X(u)$. A similar argument as above yields
\begin{align*}
\mathcal{M}_\Sigma (u)^A= \left\langle \vec{H}+\frac{X^\perp(u)}{2},\frac{\partial} {\partial y^B}\right\rangle(h_0)^{BA} e^{-|X(u)|^2/4+|X(0)|^2/4}  \sqrt{\frac{\det g_u }{\det g_0}}. 
\end{align*}
Moreover, the non-parametric form the the rescaled MCF is given by
\begin{align*} 
\frac{\partial u^A}{\partial t}-\mathcal{M}_\Sigma(u)^A=\left( e^{|X(u)|^2/4-|X(0)|^2/4} \sqrt{\frac{\det  g_0}{\det   g_u}}(h_u)^{AB}(h_0)_{BC}-\delta_C^A  \right)\cdot \mathcal{M}_\Sigma(u)^C.
\end{align*}
It is then clear that the above is of the form \eqref{equ:parabolic}.

\section*{Acknowledgement}
The first author has been partially supported by National Research Foundation of Korea grant No. 2022R1C1C1013511 and POSTECH Basic Science Research Institute grant No. 2021R1A6A1A10042944.

\bibliography{ChoiHung}

\bibliographystyle{alpha}

\end{document}